\documentclass{amsart}
\usepackage{amsmath, amsthm, amssymb}
\usepackage[mathscr]{eucal}
\usepackage{enumerate}
\usepackage[all, cmtip]{xy}
\usepackage{enumitem}
\usepackage{tikz}
\usetikzlibrary{arrows, chains, matrix, positioning, scopes}
\allowdisplaybreaks
\usepackage[colorlinks=false]{hyperref}

\newcommand{\R}{\mathbf{R}}
\newcommand{\C}{\mathbf{C}}
\newcommand{\T}{\mathbf{T}}

\newcommand{\script}[1]{{\mathcal{#1}}}
\newcommand{\A}{\script{A}}
\newcommand{\B}{\script{B}}
\newcommand{\Cgh}{\script{C}}
\newcommand{\D}{\script{D}}
\newcommand{\I}{\script{I}}
\newcommand{\E}{\mathcal{E}}
\newcommand{\Eij}{\mathcal{E}_{I,J}}
\newcommand{\EIJ}{\mathcal{E}^{I,J}}
\newcommand{\V}{\script{V}}
\newcommand{\qIJ}{q^{I,J}}
\newcommand{\qij}{q_{I,J}}

\newcommand{\Hil}{\mathcal{H}}
\newcommand{\K}{\mathcal{K}}
\newcommand{\ip}[2]{\left( \left. #1 \, \right| \, #2 \right)}
\newcommand{\hip}[2]{\langle \! \langle #1, #2 \rangle \! \rangle}
\newcommand{\hipp}[2]{\langle #1, #2 \rangle}

\newcommand{\norm}[1]{{\left\| #1 \right\|}}

\newcommand{\supp}{\operatorname{supp}}
\newcommand{\Prim}{\operatorname{Prim}}
\newcommand{\hull}{\operatorname{hull}}

\newcommand{\la}{\langle}
\newcommand{\ra}{\rangle}

\newcommand{\spa}{\operatorname{span}}

\newcommand{\go}{{G^{(0)}}}
\newcommand{\gtwo}{G^{(2)}}
\newcommand{\ho}{{H^{(0)}}}

\newcommand{\lo}{L^{(0)}}

\newcommand{\iso}{\operatorname{Iso}}
\DeclareMathOperator{\Ind}{Ind}

\newcommand{\lt}{{\rm{lt}}}

\newtheorem{prop}{Proposition}[section]
\newtheorem{thm}[prop]{Theorem}
\newtheorem{cor}[prop]{Corollary}
\newtheorem{lem}[prop]{Lemma}
\theoremstyle{definition}

\newtheorem{exmp}[prop]{Example}
\newtheorem{rem}[prop]{Remark}

\newlist{thmenum}{enumerate}{10}
\setlist[thmenum,1]{label=\textnormal{(\alph*)}}
\setlist[thmenum,2]{label=\textnormal{(\roman*)}}

\newlist{altenum}{enumerate}{10}
\setlist[altenum,1]{label=\textnormal{(\roman*)}}
\setlist[altenum,2]{label=\textnormal{(\alph*)}}

\title[Stabilization for Fell bundles over exact groupoids]{Some consequences of the stabilization theorem for Fell bundles over exact groupoids}
\author{Scott M. LaLonde}
\date{\today}
\address{Department of Mathematics, The University of Texas at Tyler, Tyler, TX 75799}
\email{slalonde@uttyler.edu}

\begin{document}

\maketitle

\begin{abstract}
	We investigate some consequences of a recent stabilization result of Ionescu, Kumjian, Sims, and Williams, which says that every Fell bundle 
	$C^*$-algebra is Morita equivalent to a canonical groupoid crossed product. First we use the theorem to give conditions that guarantee the $C^*$-algebras 
	associated to a Fell bundle are either nuclear or exact. We then show that a groupoid is exact if and only if it is ``Fell exact'', in the sense that any invariant 
	ideal gives rise to a short exact sequence of reduced Fell bundle $C^*$-algebras. As an application, we show that extensions of exact groupoids are exact
	by adapting a recent iterated Fell bundle construction due to Buss and Meyer. 
\end{abstract}

\section{Introduction}

Fell bundles over groupoids provide what is perhaps the most general notion of a groupoid action on a $C^*$-algebra. In particular, one can use Fell bundles 
to encode many different kinds of dynamical systems and $C^*$-algebraic constructions associated to groups and groupoids, including crossed products, 
twisted groupoid $C^*$-algebras and twisted crossed products, and the $C^*$-algebras of graphs and higher-rank graphs. Consequently, Fell bundles provide 
a unifying framework that allows one to study many different kinds of $C^*$-algebras simultaneously, and results generally ``trickle down'' to the various types 
of $C^*$-algebras that are modeled by them.

The ability to work in such a far-reaching setting comes at a price. There are often technical obstacles to overcome when proving results for Fell 
bundles, many of which involve delicate analyses of upper semicontinuous Banach bundles over groupoids. These issues are apparent in many of the recent
papers on the subject \cite{dana-marius,muhly-williams,aidan-dana,sims-williams2013}, which deal with amenability, ideal structure, and the much-needed
extensions of Renault's Disintegration Theorem and Renault's Equivalence Theorem to Fell bundles, among other topics.

There is some hope, however, in the form a recent stabilization trick of Ionescu, Kumjian, Sims, and Williams \cite{IKSW}. Inspired by the work of Packer 
and Raeburn on twisted crossed products \cite{packer-raeburn} and earlier work of Kumjian on \'{e}tale groupoids \cite[Corollary 4.5]{kumjian}, the authors 
constructed a canonical groupoid dynamical system $(\A, G, \alpha)$ from an arbitrary Fell bundle $p : \B \to G$ in such a way that the Fell bundle associated to
$(\A, G, \alpha)$ is equivalent to $\B$. It is an immediate consequence of this result and the equivalence theorems of \cite{muhly-williams} and \cite{sims-williams2013} 
that any Fell bundle $C^*$-algebra is Morita equivalent to a groupoid crossed product. As a result, one can now prove theorems for groupoid crossed products and 
then quickly extend them to analogous results for Fell bundles, provided the properties in question are compatible with Morita equivalence in a suitable way.

The goal of this paper is to exploit the stabilization theorem in the manner described above, with an eye toward Fell bundles over exact groupoids. We begin
with a discussion of nuclearity and exactness for Fell bundle $C^*$-algebras, with the main results following almost immediately from the stabilization theorem
and the author's previous work on groupoid crossed products \cite{lalonde2014}. More specifically, we show that if $p: \B \to G$ is a Fell bundle over an amenable
groupoid $G$ and the $C^*$-algebra $A = \Gamma_0(\go, \B)$ is nuclear, then $C^*(G, \B)$ is nuclear. Likewise, $C_r^*(G, \B)$ is exact provided $A$ is exact and 
$G$ is an exact groupoid. We then show that a groupoid $G$ is exact if and only if it is ``Fell exact'', in the sense that invariant ideals always yield short exact 
sequences of reduced Fell bundle $C^*$-algebras. We then use this result and an iterated crossed product construction of Buss and Meyer to show that extensions 
of exact groupoids are again exact.

The paper is organized as follows. We begin with some background on groupoids and Fell bundles in Section \ref{sec:prelim}. In Section 
\ref{sec:nuclear} we present our results on nuclearity and exactness, followed by a discussion on some special cases.
We shift our focus to exact groupoids in Section \ref{sec:fellexact}, and we show that a groupoid is exact if and only if it is Fell exact. Finally, Section
\ref{sec:extensions} is devoted to extensions of exact groupoids.

\section{Preliminaries}
\label{sec:prelim}

In this section we outline the necessary background information on Fell bundles over groupoids. For a groupoid $G$, we let $\go$ denote its unit space, $\gtwo$ the 
set of composable pairs, and $r, s : G \to \go$ the range and source maps, respectively. Unless otherwise specified, we assume that all groupoids 
are locally compact, Hausdorff, second countable, and come equipped with a continuous Haar system.

Let $G$ be a groupoid, and suppose $p : \B \to G$ is an upper semicontinuous Banach bundle over $G$. For each $x \in G$, we denote the fiber of $\B$ over $x$ by
$B(x)$. We say that $p : \B \to G$ is a \emph{Fell bundle} over $G$ if there is a continuous, bilinear, associative map $m : \B^{(2)} \to \B$, where
\[
	\B^{(2)} = \bigl\{ (a, b) \in \B \times \B : ( p(a), p(b) ) \in \gtwo \bigr\},
\]
and a continuous involution $b \mapsto b^*$ from $\B$ to $\B$ such that:
\begin{thmenum}
	\item $p(m(a,b)) = p(a)p(b)$ for all $(a, b) \in \B^{(2)}$;
	\item $p(b^*) = p(b)^{-1}$ for all $b \in \B$;
	\item $m(a,b)^* = m(b^*, a^*)$ for all $(a, b) \in \B^{(2)}$;
	\item for each $u \in \go$, $B(u)$ is a $C^*$-algebra with respect to the operations inherited from $\B$;
	\item For each $x \in G$, $B(x)$ is a $B(r(x))-B(s(x))$-imprimitivity bimodule with respect to the module actions induced from
		$m$ and the inner products
		\[
			{_{B(r(x))}\langle}a,b \rangle = m(a,b^*) \quad \text{and} \quad \langle a, b \rangle_{B(s(x))} = m(a^*, b).
		\]
\end{thmenum}
Since the map $m$ represents a partially-defined multiplication on $\B$, we will generally suppress it and simply write $ab$ in place of $m(a, b)$. We will also 
frequently use the shorthand $s(b)$ and $r(b)$ for $b \in \B$ to mean $s(p(b))$ and $r(p(b))$, respectively.

Since the fibers over units are $C^*$-algebras, $p : \B \vert_\go \to \go$ is an upper semicontinuous $C^*$-bundle. Consequently, the section algebra 
$A = \Gamma_0(\go, \B \vert_\go)$ is also a $C^*$-algebra. We will refer to $A$ as the \emph{$C^*$-algebra over the unit space}, or more simply, the 
\emph{unit $C^*$-algebra} of $\B$. Given the special nature of the fibers over units, we will write $A(u)$ when thinking of the fiber as a $C^*$-algebra, 
and $B(u)$ when we want to emphasize its role as an $A(u)-A(u)$-imprimitivity bimodule.

It is worth noting that the Fell bundle axioms guarantee $B(x) B(y) \subseteq B(xy)$ whenever $(x, y) \in \gtwo$. In fact, axiom (e) guarantees that multiplication
induces an isomorphism $B(x) \otimes_{A(s(x))} B(y) \cong B(xy)$ by \cite[Lemma 1.2]{muhly-williams}. In other words, our Fell bundles are always \emph{saturated}.
We also assume that all Fell bundles are separable, in the sense that $B(x)$ is a separable Banach space for all $x \in G$.

\begin{exmp}
\label{exmp:gpoidDS}
	There is one example of a Fell bundle that will be crucial throughout the paper. Let $(\A, G, \alpha)$ be a separable groupoid dynamical system, meaning
	that $\A$ is an upper semicontinuous $C^*$-bundle over $\go$ upon which $G$ acts via fiberwise isomorphisms $\alpha_x : A(s(x)) \to A(r(x))$. In order to
	build a Fell bundle, we need an upper semicontinuous Banach bundle over $G$. The natural choice is the pullback bundle $\B = r^*\A$. Note that for each 
	$x \in G$ the fiber $B(x)$ is naturally isomorphic to $B(r(x)) = A(r(x))$. We define the multiplication on $\B^{(2)}$ by
	\[
		(a, x)(b, y) = \bigl( a \alpha_x(b), xy \bigr),
	\]
	and the involution is given by
	\[
		(a, x)^* = \bigl( \alpha_x^{-1}(a^*), x^{-1} \bigr).
	\]
	It is then easy to check that axioms (a), (b), and (c) in the definition of a Fell bundle are satisfied. (See Example 2.1 of \cite{muhly-williams}.) Axiom (d) is 
	automatic, though one does need to verify that the natural operations on $A(u)$ line up with those inherited from $\B$ for all $u \in \go$. Finally, for all $x \in G$ 
	we have $B(x)=B(r(x))$, hence $B(x) \cong B(s(x))$ via $\alpha_x^{-1}$. Thus $B(x)$ is naturally a $A(r(x))-A(s(x))$-imprimitivity bimodule, and one easily checks 
	that the module actions and inner products agree with the ones inherited from $\B$. Thus $\B$ is the total space of a Fell bundle which encodes the dynamical 
	system $(\A, G, \alpha)$.
\end{exmp}

Given a Fell bundle $p : \B \to G$, we can turn the set $\Gamma_c(G, \B)$ of continuous, compactly supported sections into a convolution algebra as follows: if 
$\{\lambda^u\}_{u \in \go}$ denotes the Haar system on $G$ and $f, g \in \Gamma_c(G, \B)$, we set
\[
	f*g(y) = \int_G f(x) g(x^{-1}y) \, d\lambda^u(x).
\]
We can also define an involution on $\Gamma_c(G, \B)$ by
\[
	f^*(x) = \bigl( f(x^{-1}) \bigr)^*.
\]
One then equips $\Gamma_c(G, \B)$ with a \emph{universal norm} via $\norm{f} = \sup \norm{L(f)}$, where $L$ ranges over all $*$-representations of 
$\Gamma_c(G, \B)$ on Hilbert space that are bounded with respect to the $I$-norm (equation (1.3) of \cite{muhly-williams}). The associated completion is called
the \emph{Fell bundle $C^*$-algebra} of $\B$, denoted by $C^*(G, \B)$.

There is also a reduced norm on $\Gamma_c(G, \B)$, which is defined via \emph{regular representations}. A detailed treatment of induced representations for
Fell bundle $C^*$-algebras can be found in Section 4.1 of \cite{sims-williams2013}, so we present only the necessary details here. Let $A = \Gamma_0(\go, \B \vert_\go)$ 
be the unit $C^*$-algebra, and suppose $\pi : A \to B(\Hil)$ is a nondegenerate representation. Put $\mathsf{X}_0 = \Gamma_c(G, \B)$. Then $\mathsf{X}_0$ is a
right pre-Hilbert $A$-module under the action
\[
	(f \cdot a)(x) = f(x) a(s(x))
\]
and inner product
\[
	\hipp{f}{g}_A(u) = \int_G f(x)^* g(x) \, d\lambda_u(x).
\]
We use $\mathsf{X}$ to denote the Hilbert $A$-module obtained by completing $\mathsf{X}_0$. Note that $\Gamma_c(G, \B)$ acts on $\mathsf{X}_0$ by left convolution:
\[
	(f \cdot g)(x) = \int_G f(y) g(y^{-1} x) \, d\lambda^u(x).
\]
This action extends to an action of $C^*(G, \B)$ on $\mathsf{X}$ by adjointable operators. The induced representation $\Ind \pi$ then acts on the completion of 
$\mathsf{X} \odot \Hil$ with respect to the inner product
\[
	\ip{\xi \otimes h}{\zeta \otimes h} = \ip{\pi \bigl( \hipp{\zeta}{\xi}_A \bigr) h}{k}
\]
by
\[
	(\Ind \pi)(f)(\xi \otimes h) = (f \cdot \xi) \otimes h.
\]
If $\pi$ is taken to be faithful, then
\[
	\norm{f}_r = \norm{\Ind \pi(f)}
\]
defines a norm on $\Gamma_c(G, \B)$, called the \emph{reduced norm}. The resulting completion is the \emph{reduced Fell bundle $C^*$-algebra}, denoted by
$C_r^*(G, \B)$.

\begin{rem}
	If $p : \B \to G$ is the Fell bundle associated to a groupoid dynamical system as in Example \ref{exmp:gpoidDS}, then $C^*(G, \B) \cong \A \rtimes_\alpha G$
	\cite[Example 2.8]{muhly-williams} and $C_r^*(G, \B) \cong \A \rtimes_{\alpha, r} G$ \cite[Example 11]{sims-williams2013}.
\end{rem}

The final concept we will need is that of an equivalence between Fell bundles. Let $G$ and $H$ be groupoids endowed with Haar systems $\{\lambda_G^u\}_{u \in \go}$ 
and $\{\lambda_H^u\}_{u \in \ho}$, respectively. We say $G$ and $H$ are \emph{equivalent} if there is a locally compact Hausdorff space $Z$ such that:
\begin{itemize}
	\item $G$ and $H$ act freely and properly on the left and right of $Z$, respectively;
	\item the actions of $G$ and $H$ commute; and
	\item the anchor maps $r_Z : Z \to \go$ and $s_Z : Z \to \ho$ for the actions induce homeomorphisms $Z/H \cong \go$ and $G \backslash Z \cong \ho$.
\end{itemize}
In this case we say that $Z$ is a \emph{$G-H$-equivalence}. Note that any groupoid $G$ is equivalent to itself via $Z=G$.

Suppose $p_\B : \B \to G$ and $p_\D : \D \to H$ are Fell bundles. A \emph{$\B-\D$-equivalence} consists of a $G-H$-equivalence $Z$ and an upper semicontinuous 
Banach bundle $q : \E \to Z$ such that:
\begin{thmenum}
	\item there are commuting left and right actions of $\B$ and $\D$, respectively, on $\E$;
		
	\item there are continuous sesquilinear maps $(e, f) \mapsto {_\B \langle} e, f \rangle$ from $\E *_{s_Z} \E$ to $\B$ and $(e,f) \mapsto \langle e, f \rangle_\D$ 
		from $\E *_{r_Z} \E$ to $\D$ such that
		\begin{thmenum}
			\item $p_G \left( {_\B \langle} e, f \rangle \right) = {_G [}p(e), p(f)]$ and $p_H \left( \langle e, f \rangle_\Cgh \right) = [p(e), p(f)]_H$
			\item ${_\B \la} e, f \ra^* = {_\B \la} f, e \ra$ and $\la e, f \ra_\D^* = \la f, e \ra_\D$
			\item ${_\B \la} b \cdot e, f \ra = b \cdot {_\B \la} e, f \ra$ and $\la e, f \cdot c \ra_\D = \la e, f \ra_\D \cdot c$
			\item ${_\B \la} e, f \ra \cdot h = e \cdot \la f, h \ra_\D$
		\end{thmenum}
			
	\item for all $z \in Z$, $\E(z)$ is a $\B(r(z))-\Cgh(s(z))$-imprimitivity bimodule with respect to the operations defined in (b).
\end{thmenum}
If $q : \E \to Z$ is a $\B-\D$-equivalence, then $\mathsf{X}_0 = \Gamma_c(Z, \E)$ is a $C^*(G, \B) - C^*(H, \D)$-pre-imprimitivity bimodule with respect to the 
following operations:
\begin{align*}
	f \cdot \xi(z) &= \int_G f(x) \xi(x^{-1} \cdot z) \, d\lambda_G^{r(z)}(x) \\
	\xi \cdot g(z) &= \int_H \xi(z \cdot y) g(y^{-1}) \, d\lambda_H^{s(z)}(y) \\
	{_{C^*(G, \B)} \hip{\xi}{\eta}} &= \int_H {_\B \la} \xi(x\cdot z \cdot y), \eta(z\cdot y) \ra \, d\lambda_H^{s(z)}(y) \\
	\hip{\xi}{\eta}_{C^*(H, \D)} &= \int_G \la \xi(x^{-1} \cdot z), \eta(x^{-1} \cdot z \cdot y) \ra_{\Cgh} \, d\lambda_G^{r(z)}(x).
\end{align*}
Consequently, $C^*(G,\B)$ and $C^*(H,\D)$ are Morita equivalent; this is Renault's equivalence theorem for Fell bundles \cite[Theorem 6.4]{muhly-williams}.

Sims and Williams later showed in \cite[Theorem 14]{sims-williams2013} that $C_r^*(G, \B)$ and $C_r^*(H, \D)$ are Morita equivalent as well. They did so
by constructing a linking Fell bundle $L(\E)$ over the linking groupoid $L$ associated to the $G-H$-equivalence $Z$. They then exhibited complementary
full multiplier projections $\mathfrak{p}_G$ and $\mathfrak{p}_H$ satisfying 
\[
	\mathfrak{p}_G C_r^*(L, L(\E)) \mathfrak{p}_G \cong C_r^*(G, \B), \quad \mathfrak{p}_H C_r^*(L, L(\E)) \mathfrak{p}_H \cong C_r^*(H, \D).
\]
Hence $C_r^*(G, \B)$ and $C_r^*(H, \D)$ sit inside $C_r^*(L, L(\E))$ as complementary full corners, so $C_r^*(L, L(\E))$ serves as a linking algebra 
implementing the Morita equivalence. Indeed, $\Gamma_c(Z, \E)$ sits naturally inside $\Gamma_c(L, L(\E))$, and the imprimitivity bimodule 
$\mathfrak{p}_G C_r^*(L, L(\E)) \mathfrak{p}_H$ coming from the linking algebra is precisely the completion of $\Gamma_c(Z, \E)$ inside $C_r^*(L, L(\E))$. 

This Sims-Williams construction also works at the level of the full Fell bundle $C^*$-algebras, so $C^*(L, L(\E))$ serves as a linking algebra for $C^*(G, \B)$ and 
$C^*(H, \D)$. Their argument also shows that the $C_r^*(G, \B) - C_r^*(H, \D)$-imprimitivity bimodule $\mathfrak{p}_G C_r^*(L, L(\E)) \mathfrak{p}_H$ 
can be obtained by taking the quotient of $\mathfrak{p}_G C^*(L, L(\E)) \mathfrak{p}_H$ by the closed submodule corresponding to the kernels of the quotient
maps $C^*(G, \B) \to C_r^*(G, \B)$ and $C^*(H, \D) \to C_r^*(H, \D)$. At the risk of being overly pedantic, we now check that this construction is compatible 
with the one from \cite{muhly-williams}, so we can safely work with quotients of the usual $C^*(G, \B) - C^*(H, \D)$-imprimitivity bimodule.

\begin{prop}
\label{prop:bimoduleiso}
	The inclusion of $\Gamma_c(G, Z)$ into $\Gamma_c(L, L(\E))$ extends to a bimodule isomorphism of the $C^*(G, \B) - C^*(H, \D)$-imprimitivity bimodule 
	$\mathsf{X}$ of \cite{muhly-williams} onto the imprimitivity bimodule $\mathfrak{p}_G C^*(L, L(\E)) \mathfrak{p}_H$ of \cite{sims-williams2013}. Consequently, 
	the quotient of $\mathsf{X}$ by the closed submodule corresponding to the kernel of the quotient map $C^*(G, \B) \to C_r^*(G, \B)$ is a 
	$C_r^*(G, \B) - C_r^*(H, \D)$-imprimitivity bimodule.
\end{prop}
\begin{proof}
	It is clear from \cite{sims-williams2013} that $\mathfrak{p}_G \Gamma_c(L, L(\E)) \mathfrak{p}_H = \Gamma_c(Z, \E)$, so $\mathfrak{p}_G C^*(L, L(\E)) \mathfrak{p}_H$ 
	is the completion of $\Gamma_c(Z, \E)$ with respect to the universal norm on $C^*(L, L(\E))$. 
	Therefore, it suffices to show that the inclusion of $\Gamma_c(Z, \E)$ into $\Gamma_c(L, L(\E))$ is isometric and respects the module actions and inner 
	products.
	
	Let $\{\sigma_Z^u\}_{u \in \go}$ be the family of Radon measures on $Z$ defined in equation (2.1) of \cite{sims-williams2012}, and let $\{\kappa^u\}_{u \in \lo}$ 
	denote the associated Haar system on $L$, as defined in \cite[Lemma 2.2]{sims-williams2012}. Suppose $f \in \Gamma_c(G, \B)$ and $\xi \in \Gamma_c(Z, \E)$, 
	and view both as elements of $\Gamma_c(L, L(\E))$. Then inside $\Gamma_c(L, L(E))$, we have
	\[
		f * \xi(z) = \int_L f(x) \xi(x^{-1} \cdot z) \, d\kappa^{r(z)}(x) = \int_G f(x) \xi(x^{-1} \cdot z) \, d\lambda^{r(z)}(x)
	\]
	for $z \in Z$, since the integrand is zero unless $x \in G$ and $z \in Z$. This is precisely the formula for the left $\Gamma_c(G, \B)$-action on 
	$\Gamma_c(Z, \E)$ given in \cite[Theorem 6.4]{muhly-williams}. A similar proof works for the right action.
	
	Now we turn to the inner products. 
	If $\xi, \eta \in \Gamma_c(Z, \E)$, then the $\Gamma_c(G, \B)$-valued inner product is
	\begin{align*}
		\xi * \eta^*(x) &= \int_L \xi(z) \eta^*(z^{-1} \cdot x) \, d\kappa^{r(x)}(z) \\
			&= \int_Z \xi(z) \overline{\eta(x^{-1} \cdot z)} \, d\sigma_Z^{r(x)}(z) \\
			&= \int_H \xi(z \cdot y) \overline{\eta(x^{-1} \cdot z \cdot y)} \, d\lambda^{s(z)}(y),
	\end{align*}
	where $z$ is any element of $Z$ satisfying $r(z) = r(x)$. However, for any $z \in Z$ with $r(z) = s(x)$, we have $r(x \cdot z) = r(x)$, so we can rewrite the 
	last integral as
	\begin{align*}
		\xi * \eta^*(x) &= \int_H \xi(x \cdot z \cdot y) \overline{\eta(z \cdot y)} \, d\lambda^{s(z)}(y) \\
			&= \int_H {_\B \la} \xi(x \cdot z \cdot y), \eta(z \cdot y) \ra \, d\lambda^{s(z)}(y),
	\end{align*}
	which agrees with the inner product from \cite{muhly-williams}. The proof for the $C^*(H, \D)$-valued inner product is similar. It follows that the inclusion of 
	$\Gamma_c(Z, \E)$ into $\Gamma_c(L, L(\E))$ respects the norms induced from $C^*(G, \B)$ (or from $C^*(H, \D)$). Thus the inclusion is isometric and extends to 
	an isomorphism of imprimitivity bimodules.
\end{proof}

The upshot of Proposition \ref{prop:bimoduleiso} is the ability to construct a $C_r^*(G, \B) - C_r^*(H, \D)$-imprimitivity bimodule without appealing to a linking
algebra. That is, we may work with $\Gamma_c(Z, \E)$ endowed with the operations defined in \cite{muhly-williams}, and simply complete it with respect to
the norm induced by the reduced norm on $\Gamma_c(G, \B)$.

\section{Stabilization, nuclearity and exactness}
\label{sec:nuclear}
Throughout this section, $G$ denotes a second countable, locally compact Hausdorff groupoid with Haar system $\{\lambda^u\}_{u \in \go}$, and $p : \B \to G$ is a 
separable saturated Fell bundle. We let $A = \Gamma_0(\go, \B)$ denote the unit $C^*$-algebra of $\B$.

In \cite{IKSW}, Ionescu, Kumjian, Sims, and Williams showed that the full and reduced $C^*$-algebras associated to the Fell bundle $\B$ are Morita equivalent to the 
full and reduced crossed products, respectively, coming from a canonical groupoid dynamical system. In particular, they constructed an upper semicontinuous Banach 
bundle $\V$ over $\go$ such that the following conditions hold.
\begin{itemize}
	\item Each fiber $V(u)$ is a full right Hilbert $A(u)$-module.
	\item The section algebra $V = \Gamma_0(\go, \V)$ is a full right Hilbert $A$-module.
	\item There is a natural action $\alpha$ of $G$ on $\K(\V)$, the upper semicontinuous $C^*$-bundle over $\go$ whose fibers are
		\[
			\K(\V)(u) = \K(V(u)),
		\]
		where $\K(V(u))$ denotes the set of generalized compact operators on $V(u)$.
	\item The section algebra of $\K(\V)$ can be identified with $\K(V)$, the algebra of generalized compact operators on $V$. Note that $\K(V)$ is Morita equivalent 
	to $A$ via the imprimitivity bimodule $V$.
\end{itemize}
The stabilization theorem \cite[Theorem 3.7]{IKSW} then says that there is an equivalence between $\B$ and the Fell bundle associated to the 
dynamical system $(\K(\V), G, \alpha)$. Consequently, $C^*(G, \B)$ is Morita equivalent to $\K(\V) \rtimes_\alpha G$, and $C^*_r(G, \B)$ is Morita equivalent to 
$\K(\V) \rtimes_{\alpha ,r} G$ by the equivalence theorems for full and reduced Fell bundle $C^*$-algebras \cite{muhly-williams,sims-williams2013}. 

The stabilization theorem offers the possibility that certain questions regarding Fell bundles can be answered by instead looking at the simpler case of groupoid crossed 
products. To wit, one can perhaps prove results for groupoid crossed products and then extend those results to Fell bundles via the stabilization theorem. Of course this 
line of attack is particularly effective for properties that are preserved under Morita equivalence. As our first examples, we can easily extend two of the author's previous 
results \cite[Theorems 5.1 and 6.14]{lalonde2014} for groupoid crossed products to obtain conditions that guarantee a Fell bundle $C^*$-algebra is nuclear or exact.

The following nuclearity result is already known in the special case of \emph{continuous} Fell bundles over \'{e}tale groupoids by a result of Takeishi 
\cite[Theorem 4.1]{takeishi}.

\begin{thm}
\label{thm:nuke}
	Let $G$ be a second countable locally compact Hausdorff groupoid endowed with a Haar system, $p : \B \to G$ a separable saturated Fell bundle over $G$, and 
	$A = \Gamma_0(\go, \B)$ the unit $C^*$-algebra of $\B$. If $A$ is nuclear and $G$ is measurewise amenable, then $C^*(G, \B)$ is nuclear.
\end{thm}
\begin{proof}
	Let $(\K(\V), G, \alpha)$ denote the groupoid dynamical system afforded by the stabilization theorem of \cite{IKSW}. Since $\K(V)$ and $A$ are Morita equivalent 
	and $A$ is nuclear, \cite[Theorem 15]{huefraewil} guarantees that $\K(V)$ is nuclear. Since $G$ is amenable, $\K(\V) \rtimes_\alpha G$ is nuclear by 
	\cite[Theorem 5.1]{lalonde2014}. But $C^*(G, \B)$ and $\K(\V) \rtimes_\alpha G$ are Morita equivalent by \cite[Corollary 3.8]{IKSW}, so $C^*(G, \B)$ is nuclear.
\end{proof}

A nearly identical argument gives us conditions for exactness of the reduced Fell bundle $C^*$-algebra $C_r^*(G, \B)$. First recall that if $(\A, G, \alpha)$ is a groupoid 
dynamical system, an ideal $I \subseteq A = \Gamma_0(\go, \A)$ is said to be \emph{invariant} if
\[
	\alpha_x(I(s(x)) = I(r(x))
\]
for all $x \in G$. We say $G$ is \emph{exact} if for any dynamical system $(\A, G, \alpha)$ and any invariant ideal $I \subseteq A$, the sequence
\[
	0 \to \I \rtimes_{\alpha \vert_I, r} G \to \A \rtimes_{\alpha, r} G \to \A/\I \rtimes_{\alpha^I, r} G \to 0
\]
of reduced crossed products is exact.

\begin{thm}
\label{thm:exact}
	Let $G$ be a second countable locally compact Hausdorff groupoid endowed with a Haar system, $p : \B \to G$ a separable saturated Fell bundle over $G$, and 
	$A = \Gamma_0(\go, \B)$ the unit $C^*$-algebra of $\B$. If $A$ is exact and $G$ is an exact groupoid, then $C_r^*(G, \B)$ is exact.
\end{thm}
\begin{proof}
	Assume $A$ is exact and $G$ is exact. Since $\K(V)$ and $A$ are Morita equivalent, $\K(V)$ is exact by a theorem of Katsura 
	\cite[Proposition A.10]{katsura}. If we assume $G$ is exact, then $\K(\V) \rtimes_\alpha G$ is exact by \cite[Theorem 6.14]{lalonde2014}. It then follows again from 
	\cite[Corollary 3.8]{IKSW} that $C_r^*(G, \B)$ is exact.
\end{proof}

As mentioned above, Fell bundles provide a convenient setting in which to work, since results will immediately descend to many different types of $C^*$-algebras. For 
example, Theorems \ref{thm:nuke} and \ref{thm:exact} have immediate implications for twisted groupoid $C^*$-algebras and, more generally, twisted crossed products.

\subsection{Twisted crossed products}
One construction that is subsumed by Fell bundles is the \emph{twisted groupoid crossed product}, as introduced by Renault in \cite{renault87} and described further in 
\cite{muhly-bundles,muhly-williams}. Suppose we have a central extension of groupoids
\begin{equation}
\label{eqn:ext}
	\xymatrix{
		\go \ar[r] & S \ar[r]^i & E \ar[r]^j & G \ar[r] & \go,
	}
\end{equation}
where $S$ is a bundle of abelian groups, and that $\A \to \go$ is an upper semicontinuous $C^*$-bundle. Suppose further that $E$ acts on $\A$ via a family of isomorphisms
$\alpha = \{\alpha_e\}_{e \in E}$, and that there is a homomorphism $\chi : S \to \bigsqcup_{u \in \go} M(A(u))$ implementing the resulting action of $S$. More specifically, we
assume that the map $(t, a) \mapsto \chi(t)a$ is continuous and
\[
	\alpha_t(a) = \chi(t) a \chi(t)^*, \quad \chi(ete^{-1}) = \overline{\alpha_e}(\chi(t))
\]
for all $t \in S$, $a \in A(s(t))$, and $e \in E$ with $s(e) = r(t)$. We call $(\A, G, E, \alpha)$ a \emph{twisted dynamical system}. (Note that if $S = \go$, we recover the usual
notion of a groupoid dynamical system.) Renault then considers continuous $\A$-valued functions on $E$ that have ``compact support modulo $S$'', i.e., sections 
$f : E \to r^*\A$ satisfying
\[
	f(te) = f(e) \chi(t^{-1})
\]
for all $t \in S$. Such functions form a $*$-algebra under the operations
\[
	f*g(e) = \int_G f(e') \alpha_y \bigl( g(e'^{-1}e) \bigr) \, d\lambda^{r(e)}({j(e')}), \quad f^*(e) = \alpha_e \bigl( f(e^{-1}) \bigr)^*,
\]
and we equip this $*$-algebra with a norm by taking the supremum over all appropriately bounded representations (as in the untwisted case). The completion is the 
\emph{twisted crossed product}, denoted by $C^*(G, E, \A)$.

By adapting the setup from Example \ref{exmp:gpoidDS}, we can bring twisted crossed products under the Fell bundle umbrella. Start with the bundle $r^*\A \to E$, upon
which $S$ acts naturally via
\[
	(a, e) \cdot t = (a \chi(t)^*, te).
\]
If we put $\B = (r^*\A)/S$, then Lemmas 2.6 and 2.7 of \cite{muhly-williams} show that $p: \B \to G$ is a Fell bundle, where $p([a, e]) = j(e)$ and the operations are defined
by
\[
	[a, e][b, e'] = [a \alpha_e(b), ee'], \quad [a, e] = \bigl[ \alpha_e^{-1}(a^*), e^{-1} \bigr].
\]
Furthermore, $C^*(G, \B)$ is isomorphic to the twisted crossed product $C^*(G, E, \A)$ by \cite[Example 2.10]{muhly-williams}. 

It is also fairly straightforward to check that the unit $C^*$-algebra of $\B$ is isomorphic to $A$. Notice first that if $[a, e] \in \B \vert_{\go}$, then we have $p([a, e]) = j(e) \in \go$,
so $e \in S$. This observation gives us a way of identifying the fibers of $\B$ over units.

\begin{lem}
	For each $u \in \go$, define $\varphi_u : B(u) \to A(u)$ by
	\[
		\varphi_u([a, t]) = (a \chi(t), u),
	\]
	where $t$ is any element of $S_u$. Then $\varphi_u$ defines an isomorphism of $B(u)$ onto $A(u)$.
\end{lem}
\begin{proof}
	First note that $\varphi_u$ is well-defined: if $t' \in S_u$, then $[a \chi(t')^*, t't] = [a, t]$, and
	\[
		\varphi_u \bigl( [a \chi(t')^*, t't] \bigr) = (a\chi(t')^*\chi(t't), u) = (a \chi(t), u) = \varphi_u([a, t]).
	\]
	Clearly $\varphi_u$ is linear, and we have
	\begin{align*}
		\varphi_u \bigl( [a, t][b, t'] \bigr) &= \varphi_u \bigl( [a \alpha_t(b), tt'] \bigr) \\
			&= \bigl( a \chi(t)b\chi(t)^*\chi(tt'), u \bigr) \\
			&= \bigl( a \chi(t) b \chi(t'), u \bigr) \\
			&= \varphi_u([a, t]) \varphi_u([b, t'])
	\end{align*}
	and
	\[
		\varphi_u \bigl( [a, t]^* \bigr) = \bigl( \alpha_t^{-1}(a^*) \chi(t)^*, u \bigr) = (\chi(t)^*a^*, u) = \varphi_u([a, t]).
	\]
	Thus $\varphi_u$ is a $*$-homomorphism. If $\varphi_u([a, t]) = 0$, then $a \chi(t) = 0$, so $a=0$ and $\varphi_u$ is injective. It is clearly surjective, hence an
	isomorphism.
\end{proof}

By gluing together the fiberwise homomorphisms $\varphi_u$, we obtain a bundle map $\hat{\varphi} : \B \vert_\go \to \A$ given by $\hat{\varphi} \vert_{B(u)} = \varphi_u$.
Moreover, it is not hard to check that $\hat{\varphi}$ is continuous. If $[a_i, t_i] \to [a, t]$, then we can pass to a subnet, relabel, and assume $(a_i, t_i) \to (a, t)$ in $r^*\A$.
Then $a_i \to a$ in $\A$, so $a_i \chi(t_i) \to a \chi(t)$ by the continuity requirement on $\chi$. Hence
\[
	\hat{\varphi} \bigl( [a_i, t_i] \bigr) = \bigl( a_i \chi(t_i), s(t_i) \bigr) \to \bigl( a \chi(t), s(t) \bigr) = \hat{\varphi} \bigl( [a, t] \bigr).
\]
As discussed in \cite[Remark 3.6]{lalonde2014}, the bundle map $\hat{\varphi}$ induces a $C_0(\go)$-linear isomorphism $\Phi : \Gamma_0(\go, \B \vert_\go) \to A$. With 
this fact in hand, we have the following immediate corollary of Theorem \ref{thm:nuke}.

\begin{thm}
\label{thm:nucleartwisted}
	Let $(\A, G, E, \alpha)$ be a twisted groupoid dynamical system. If $G$ is amenable and $A$ is nuclear, then the twisted crossed product $C^*(G, E, \A)$ is nuclear.
\end{thm}

We should point out that Theorem \ref{thm:nucleartwisted} is not really new. Indeed, it follows from \cite[Lemme 3.3(i)]{renault87} that $C^*(G, E, \A)$ is a quotient 
of the (untwisted) crossed product $\A \rtimes_\alpha E$. Since $S$ is an abelian group bundle (hence amenable) and amenability is preserved under taking extensions of 
groupoids \cite[Theorem 5.3.14]{ananth-renault}, it follows that $E$ is amenable whenever $G$ is. Hence $\A \rtimes_\alpha G$ is nuclear if $A$ is nuclear and $G$ is 
amenable, so $C^*(G, E, \A)$ is also nuclear under these hypotheses.

On the other hand, the special case of Theorem \ref{thm:exact} for twisted dynamical systems does appear to be new. In fact, there does not even seem to be a notion of 
reduced twisted groupoid crossed products in the literature. Therefore, for a twisted dynamical system $(\A, G, E)$ we define the \emph{reduced twisted 
crossed product} $C^*_r(G, E, \A)$ to be $C_r^*(G, \B)$, where $\B$ is the Fell bundle associated to $(\A, G, E)$.

\begin{thm}
\label{thm:twistedexact}
	Let $(\A, G, E, \alpha)$ be a twisted groupoid dynamical system. If $G$ is exact and $A$ is exact, then the reduced twisted crossed product $C_r^*(G, E, \A)$ is exact.
\end{thm}

As a special case of twisted crossed products, we can also study \emph{twists}, which were initially defined by Kumjian in \cite{kumjian}. Several other authors 
\cite{muhly-williams,mw92,mw95,renault08,erikdana} have since discussed twists over groupoids and their relationship to Fell bundles. If $G$ is a locally compact Hausdorff 
groupoid, a \emph{twist} over $G$ (sometimes called a $\T$-groupoid) is a central groupoid extension
\[
\xymatrix{
	\go \ar[r] & \go \times \T \ar[r]^(0.6){i} & \Sigma \ar[r]^j & G \ar[r]& \go.
}
\]
In other words, we take $S = \go \times \T$ in \eqref{eqn:ext}. It is worth noting that any continuous cocycle $\omega : \gtwo \to \T$ gives rise to a twist, though the discussion 
in \cite[Section 2]{mw92} shows that the theory of twists is more general. 

The authors of \cite{mw92} describe how to directly construct the $C^*$-algebra $C^*(G, \Sigma)$ associated to a twist, but one can also realize $C^*(G, \Sigma)$ 
as a twisted crossed product. Start with the trivial dynamical system $(\C \times \go, \Sigma, \lt)$, and define $\chi : \go \times \T \to \go \times \C$ by $\chi(t) = t$. We obtain 
a twisted dynamical system $(\C \times \go, G, \Sigma, \lt)$, and it is not hard to check that the associated Fell bundle agrees with the one built in \cite[Example 2.3]{erikdana} 
and \cite[Section 4]{renault08}. Hence $C^*(G, \Sigma, \C \times \go) = C^*(G, \Sigma)$ and $C_r^*(G, \Sigma, \C \times \go) = C_r^*(G, \Sigma)$. In light of this discussion, 
Theorems \ref{thm:nuke} and \ref{thm:exact} yield the following results for twists.

\begin{thm}
\label{thm:nucleartwist}
	Let $G$ be a second countable locally compact Hausdorff groupoid, and suppose $\Sigma$ is a twist over $G$.
	\begin{enumerate}
		\item If $G$ is measurewise amenable, then $C^*(G, \Sigma)$ is nuclear.
		\item If $G$ is exact, then $C^*_r(G, \Sigma)$ is exact.
	\end{enumerate}
\end{thm}

Again, statement (1) in Theorem \ref{thm:nucleartwist} is already known, since $C^*(G, \Sigma)$ is a quotient of the groupoid $C^*$-algebra $C^*(\Sigma)$.

\section{Fell exact groupoids}
\label{sec:fellexact}

We now turn our attention toward a more refined analysis of Fell bundles over exact groupoids. In particular, we will use the stabilization theorem to show that if $G$ is 
exact, then an invariant ideal in the unit $C^*$-algebra of any Fell bundle over $G$ gives rise to a short exact sequence of reduced Fell bundle $C^*$-algebras.

Let $G$ be a second countable, locally compact Hausdorff groupoid, and suppose $p_\B : \B \to G$ is a separable Fell bundle over $G$ with $A = \Gamma_0(\go, \B)$. 
While $G$ does not necessarily act on the $C^*$-bundle $\B \vert_\go$ associated to $A$, it does act naturally on $\Prim A$ as follows. We identify $\Prim A$ with the 
disjoint union $\bigsqcup_{u \in \go} \Prim A(u)$ via \cite[Proposition C.5]{TFB1}, and for each $x \in G$ the Rieffel correspondence associated to $B(x)$ induces 
a homeomorphism $h_x : \Prim A(s(x)) \to A(r(x))$. We then set
\begin{equation}
\label{eq:primaction}
	x \cdot (s(x), P) = (r(x), h_x(P)).
\end{equation}
It is shown in \cite[Proposition 2.2]{dana-marius} that \eqref{eq:primaction} defines a continuous action of $G$ on $\Prim A$. We then say an ideal $I \subseteq A$ is 
\emph{invariant} if
\[
	\hull(I) = \{ P \in \Prim A : I \subseteq P \}
\]
is a $G$-invariant subset of $\Prim A$. 

If the Fell bundle $p_\B : \B \to G$ does come from a groupoid dynamical system, then we have two competing notions of invariance for ideals.
However, it is straightforward to check that the two definitions are equivalent in this case.

\begin{prop}
	Let $(\A, G, \alpha)$ be a groupoid dynamical system. An ideal $I \subseteq A$ is invariant if and only if $\hull(I)$ is a $G$-invariant subset of $\Prim A$.
\end{prop}
\begin{proof}
	Suppose first that $\hull(I)$ is $G$-invariant, and let $x \in G$. Then by definition $\hull(I(r(x))) = \alpha_x \bigl( \hull(I(s(x))) \bigr)$, and since $I(r(x))$ is equal to 
	the intersection of all the primitive ideals containing it, we have
	\[
		I(r(x)) = \bigcap_{P \in \hull(I(s(x)))} \alpha_x(P) = \alpha_x \Biggl( \bigcap_{P \in \hull(I(s(x)))} P \Biggr) 
			= \alpha_x(I(s(x))).
	\]
	On the other hand, suppose $I$ is invariant, and let $P \in \hull(I)$. Identify $P$ with the pair $(s(x), P)$, where $P \in \Prim A(s(x))$. We have 
	$I(s(x)) \subseteq P$, so
	\[
		I(r(x)) = \alpha_x(I(s(x))) \subseteq \alpha_x(P).
	\]
	Thus $(r(x), \alpha_x(P))$ belongs to $\hull(I)$. Hence $\hull(I)$ is $G$-invariant.
\end{proof}

It is shown in \cite{dana-marius} that if $I$ is an invariant ideal, then there are Fell bundles $\B_I$ and $\B^I$ over $G$ with $I = \Gamma_0(\go, \B_I)$ and 
$A/I = \Gamma_0(\go, \B^I)$. Furthermore, there is a short exact sequence of Fell bundle $C^*$-algebras
\[
	0 \to C^*(G, \B_I) \to C^*(G, \B) \to C^*(G, \B^I) \to 0
\]
by \cite[Theorem 3.7]{dana-marius}. The same cannot cannot be said for the reduced $C^*$-algebras, since exactness is known to fail for reduced crossed 
products associated to certain groupoids (or even some groups). Therefore, we will focus on the sequence of reduced Fell bundle $C^*$-algebras
\begin{equation}
\label{eq:ses}
	0 \to C_r^*(G, \B_I) \to C_r^*(G, \B) \to C_r^*(G, \B^I) \to 0,
\end{equation}
and attempt to determine when it is guaranteed to be exact. In light of the stabilization theorem, one might guess that it suffices to require $G$ to be exact.

Before proceeding any further, we first need to make sure that sequences like the one in \eqref{eq:ses} actually make sense. That is, we need to verify that there is a 
natural inclusion $C_r^*(G, \B_I) \hookrightarrow C_r^*(G, \B)$ and a surjection $C_r^*(G, \B) \to C_r^*(G, \B^I)$. In \cite[Lemma 3.5]{dana-marius}, the authors 
show that the inclusion $I \subseteq A$ (and subsequent embedding of $\B_I$ into $\B$) yields an natural inclusion
\[
	\iota : \Gamma_c(G, \B_I) \hookrightarrow \Gamma_c(G, \B)
\]
which extends to an isomorphism of $C^*(G, \B_I)$ onto an ideal of $C^*(G, \B)$. We desire an analogous result for reduced $C^*$-algebras, so we need to 
show that $\iota$ is isometric with respect to the reduced norms on $\Gamma_c(G, \B_I)$ and $\Gamma_c(G, \B)$.

\begin{prop}
\label{prop:ideal}
	The inclusion map $\iota : \Gamma_c(G, \B_I) \hookrightarrow \Gamma_c(G, \B)$ extends to an isomorphism of $C_r^*(G, \B_I)$ onto an ideal of $C_r^*(G, \B)$.
\end{prop}
\begin{proof}
	The spirit of the proof is similar to that of \cite[Lemma 6.9]{lalondeexact}. Let $\pi : A \to \B(\Hil_\pi)$ be a faithful representation on a separable Hilbert space $\Hil_\pi$, 
	and form the associated regular representation $\Ind \pi$ of $C^*(G, \B)$. Then for any $f \in \Gamma_c(G, \B_I)$,
	\[
		\norm{\iota(f)}_r = \norm{\Ind \pi(\iota(f))}.
	\]
	We are tempted at this point to say that $(\Ind \pi) \vert_{C^*(G,\B_I)} = \Ind \pi \vert_I$, so that the above norm is just $\norm{f}_r$. However, the representations
	$(\Ind \pi) \vert_{C^*(G,\B_I)}$ and $\pi \vert_I$ might be degenerate, which complicates the matter.
	
	To work around these issues, let $\Hil$ denote the essential subspace of $\pi \vert_I$. Then $\pi \vert_I$ is a faithful, nondegenerate representation of $I$ on $\Hil$. 
	Also, the subspace $\Hil \subseteq \Hil_\pi$ is invariant for all the operators in $\pi(A)$, so we obtain a nondegenerate subrepresentation $\rho$ of $A$ on $\Hil$. 
	Notice that $\rho \vert_I$ is faithful, since $\pi \vert_I = \rho \vert_I \oplus 0$. Moreover, $\rho$ is faithful on $A$: if $a \in A$ and $b \in I$, then
	\[
		\rho(a)\rho(b) = \rho(ab) \neq 0,
	\]
	since $ab \in I$ and $\rho \vert_I$ is faithful. Thus $\rho(a) \neq 0$ and $\rho$ is faithful. Therefore, if we form the induced representation $\Ind \rho$ of $C^*(G, \B)$ on
	$\mathsf{X} = \overline{\Gamma_c(G, \B) \odot \Hil}$, then $\norm{f}_r = \norm{\Ind \rho(f)}$ for all $f \in \Gamma_c(G, \B)$. On the other hand, we can form the induced 
	representation $\Ind \rho \vert_I$ on $\mathsf{X}_I = \overline{\Gamma_c(G, \B_I) \odot \Hil}$, and for all $f \in \Gamma_c(G, \B_I)$,
	\[
		\norm{f}_r = \norm{\Ind \rho \vert_I (f)}.
	\]
	Now observe that $\Gamma_c(G, \B_I) \odot \Hil$ sits naturally inside $\Gamma_c(G, \B) \odot \Hil$, and this embedding is isometric: for all 
	$\xi \in \Gamma_c(G, \B_I)$ and $h \in \Hil$, we have
	\[
		\ip{\xi \otimes h}{\xi \otimes h}_{\mathsf{X}} = \ip{\rho \bigl( \la \xi, \xi \ra_A \bigr) h}{h} 
			= \ip{\pi \vert_I \bigl( \la \xi, \xi \ra_I \bigr) h}{h} 
			= \ip{\xi \otimes h}{\xi \otimes h}_{\mathsf{X}_I}.
	\]
	Thus $\mathsf{X}_I$ embeds isometrically into $\mathsf{X}$. We claim that $\mathsf{X}_I$ is the essential subspace for the possibly degenerate representation
	$(\Ind \rho) \vert_{C^*(G, \B_I)}$. To see this, let $\xi \in \Gamma_c(G, \B)$, $h \in \Hil$, and $f \in \Gamma_c(G, \B_I)$. Then
	\[
		\Ind \rho(f) (\xi \otimes h) = f \cdot \xi \otimes h,
	\]
	where $f \cdot \xi = f * \xi \in \Gamma_c(G, \B_I)$, since $\Gamma_c(G, \B_I)$ is an ideal in $\Gamma_c(G, \B)$. Thus $f \cdot \xi \otimes h \in \Gamma_c(G, \B_I) \odot \Hil$.
	It follows that
	\[
		\overline{\spa} \{ \Ind \rho(f) x : f \in C^*(G, \B_I), x \in \mathsf{X} \} = \mathsf{X}_I,
	\]
	and the left hand side is precisely the essential subspace for $(\Ind \rho) \vert_{C^*(G, \B_I)}$. It is then clear that $(\Ind \rho) \vert_{C^*(G, \B_I)} \circ \iota$ and 
	$\Ind \rho\vert_I$ agree on $\mathsf{X}_I$. It now follows that for all $f \in \Gamma_c(G, \B_I)$,
	\[
		\norm{\iota(f)}_r = \norm{(\Ind \rho) \vert_{C^*(G, \B_I)}(\iota(f))} = \norm{\Ind \rho \vert_I (f)} = \norm{f}_r.
	\]
	Therefore, $\iota$ is isometric for the reduced norms, so it extends to an isomorphism of $C_r^*(G, \B_I)$ onto an ideal of $C_r^*(G, \B)$.
\end{proof}

It is shown in \cite[Lemma 3.6]{dana-marius} that the quotient homomorphism $q_I : A \to A/I$ induces a natural surjective homomorphism 
$q : \Gamma_c(G, \B) \to \Gamma_c(G, \B^I)$ via the fiberwise quotient maps $q_x : B(x) \to B^I(x)$, which extends to a surjection of $C^*(G, \B)$ onto $C^*(G, \B^I)$. 
We will now verify the analogous result for the reduced Fell bundle $C^*$-algebras. Fortunately, the proof does not require the same machinations with degenerate 
representations that were necessary for the previous proposition.

\begin{prop}
	The map $q : \Gamma_c(G, \B) \to \Gamma_c(G, \B^I)$ extends to a surjective homomorphism $q : C_r^*(G, \B) \to C_r^*(G, \B^I)$.
\end{prop}
\begin{proof}
	Since the authors of \cite{dana-marius} already verified that $q : \Gamma_c(G, \B) \to \Gamma_c(G, \B^I)$ is a $*$-homomorphism, it only remains to show 
	that $q$ is bounded with respect to the reduced norms. Let $\pi : A/I \to B(\Hil)$ be a faithful representation, and form the induced representation $\Ind \pi$ of 
	$C_r^*(G, \B^I)$ on $\mathsf{X}^I = \overline{\Gamma_c(G, \B^I) \odot \Hil}$. On the other hand, $\pi \circ q_I$ is a nondegenerate representation of $A$ on 
	$\Hil$, and the associated induced representation $\Ind(\pi \circ q_I)$ of $C_r^*(G, \B)$ acts on $\mathsf{X} = \overline{\Gamma_c(G, \B) \odot \Hil}$. Define 
	$U_0 : \Gamma_c(G, \B) \odot \Hil \to \Gamma_c(G, \B^I) \odot \Hil$ by
	\[
		U_0(\xi \otimes h) = q(\xi) \otimes h.
	\]
	We claim that $U_0$ extends to a unitary $U : \mathsf{X} \to \mathsf{X}^I$. To see this, observe that if $\xi, \eta \in \Gamma_c(G, \B)$ and $h, k \in \Hil$, then
	\begin{align*}
		\ip{U_0(\xi \otimes h)}{U_0(\eta \otimes k)} &= \ip{q(\xi) \otimes h}{q(\eta) \otimes k} \\
			&= \ip{\pi \bigl( \la q(\eta), q(\xi) \ra_{A/I} \bigr) h}{k},
	\end{align*}
	where
	\begin{align*}
		\la q(\eta), q(\xi) \ra_{A/I}(u) &= \int_G q(\eta)(x)^* q(\xi)(x) \, d\lambda_u(x) \\
			&= \int_G q_{x^{-1}}(\eta(x))^* q_x(\xi(x)) \, d\lambda_u(x) \\
			&= \int_G q_{s(x)} \bigl( \eta(x)^* \xi(x) \bigr) \, d\lambda_u(x) \\
			&= q_u \left( \int_G \eta(x)^* \xi(x) \, d\lambda_u(x) \right) \\
			&= q_u \bigl( \la \eta, \xi \ra_A (u) \bigr).
	\end{align*}
	Therefore,
	\[
		\ip{U_0(\xi \otimes h)}{U_0(\eta \otimes k)} = \ip{ (\pi \circ q_I) \bigl( \la \eta, \xi \ra_A \bigr) h}{k} = \ip{\xi \otimes h}{\eta \otimes k},
	\]
	so $U_0$ is isometric. It is clear that $U_0$ has dense range, thus it extends to a unitary $U : \mathsf{X} \to \mathsf{X}^I$. Furthermore, $U$ intertwines  
	$\Ind(\pi \circ q_I)$ and $(\Ind \pi) \circ q$:
	\begin{align*}
		\Ind \pi (q(f)) U(\xi \otimes h) &= q(f) \cdot q(\xi) \otimes h \\
			&= q(f \cdot \xi) \otimes h \\
			&= U(f \cdot \xi \otimes h) \\
			&= U \bigl( \Ind(\pi \circ q_I)(f)(\xi \otimes h) \bigr)
	\end{align*}
	for all $f, \xi \in \Gamma_c(G, \B)$ and $h \in \Hil$. Therefore, for all $f \in \Gamma_c(G, \B)$ we have
	\[
		\norm{q(f)}_r = \norm{\Ind \pi (q(f))} = \norm{\Ind(\pi \circ q_I)(f)} \leq \norm{f}_r.
	\]
	Hence $q$ is norm-decreasing, so it extends to a homomorphism $q : C_r^*(G, \B) \to C_r^*(G, \B^I)$, which is surjective since $q$ has dense range.
\end{proof}

Now we proceed with determining when sequences like the one in \eqref{eq:ses} are exact. Instead of working directly with the dynamical system
afforded by the Stabilization Theorem, the details are a little nicer if we work in a slightly more abstract setting at first. Let $p_\D : \D \to G$ be another Fell 
bundle over $G$, and let $C = \Gamma_0(G, \D)$ denote its unit $C^*$-algebra. Furthermore, suppose $q : \E \to G$ is a $\B-\D$-equivalence over the trivial 
$G-G$-equivalence $G$. It is then straightforward to check that the restriction $\E \vert_{\go}$ is a $\B \vert_{\go} - \D \vert_{\go}$-imprimitivity bimodule bundle 
(as defined in \cite[Definition 2.17]{KMRW} and discussed further in \cite[Definition 6.14]{jonbrown}), so $E = \Gamma_0(\go, \E)$ is an $A-C$-imprimitivity 
bimodule. We let $h : \mathcal{I}(C) \to \mathcal{I}(A)$ denote the associated Rieffel correspondence between the ideal lattices of $C$ and $A$, respectively.

Now suppose $J \subseteq C$ is an invariant ideal, and let $I = h(J)$ be the corresponding ideal in $A$. We intend to prove that the sequence 
\begin{equation}
\label{eq:sesB}
	0 \to C_r^*(G, \B_I) \to C_r^*(G, \B) \to C_r^*(\B^I) \to 0
\end{equation}
is exact if and only if
\begin{equation}
\label{eq:sesD}
		0 \to C_r^*(G, \D_J) \to C_r^*(G, \D) \to C_r^*(G, \D^J) \to 0
\end{equation}
is exact. Of course we first need to know that $I$ is an invariant ideal for this conjecture to even make sense. Before we begin the proof, it will be helpful to introduce 
some additional notation. For each $x \in G$ we let
\[
	h_x^{A} : \I(A(s(x))) \to \I(A(r(x))), \quad h_x^C : \I(C(s(x))) \to \I(C(r(x)))
\]
denote the Rieffel correspondences coming from the imprimitivity bimodules $B(x)$ and $D(x)$, respectively. Also, note that for each $u \in \go$, the Rieffel 
correspondence $h : \I(C) \to \I(A)$ descends to a bijection $h_u : \I(C(u)) \to \I(A(u))$ \`{a} la \cite[Remark 3.26]{TFB1}. It is then straightforward to check that 
we have a commuting diagram:
\[
	\xymatrix{
		\I(A(r(x))) &  \I(C(r(x))) \ar[l]_{h_{r(x)}} \\
		\I(A(s(x))) \ar[u]^{h_x^A} & \I(C(s(x))) \ar[u]_{h_x^C} \ar[l]_{h_{s(x)}}
	}
\]
Indeed, this diagram commutes thanks to \cite[Lemma 6.2]{muhly-williams}, which guarantees that we have natural isomorphisms
\[
	B(x) \otimes_{A(s(x))} E(s(x)) \cong E(x \cdot s(x)) = E(x)
\]
and
\[
	E(r(x)) \otimes_{C(r(x))} D(x) \cong E(r(x) \cdot x) = E(x)
\]
of $A(r(x))-C(s(x))$-imprimitivity bimodules.

\begin{prop}
	Let $J \subseteq C$ be an invariant ideal, and let $I = h(J)$ be the corresponding ideal in $A$. Then $I$ is invariant.
\end{prop}
\begin{proof}
	We need to show that $\hull(I)$ is a $G$-invariant subset of $\Prim A$. First observe that
	\begin{align*}
		\hull(I) &= \{P \in \Prim A : I \subseteq P \} \\
			&= \{ h(Q) \in \Prim A : I=h(J) \subseteq h(Q) \} \\
			&= h \bigl( \{ Q \in \Prim C : J \subseteq Q \} \bigr) \\
			&= h(\hull(J)).
	\end{align*}
	Now let $x \in G$ and suppose $P \in \hull(I)$ is lifted from the fiber $A(s(x))$, so we can identify $P$ with $(s(x), P)$. As in the proof of Corollary 3.9 
	of \cite{IKSW}, we have
	\[
		x \cdot P = h(x \cdot h^{-1}(P))
	\]
	since
	\begin{align*}
		x \cdot (s(x), P) &= (r(x), h_x^{A}(P)) \\
		&= \bigl( r(x), h_{r(x)}(h_x^C(h_{s(x)}^{-1}(P))) \bigr) \\
		&= h \bigl( r(x), h_x^C(h_{s(x)}^{-1}(P)) \bigr) \\
		&= h(x \cdot h^{-1}(P)).
	\end{align*}
	Since $P \in \hull(I)$, $h^{-1}(P) \in \hull(J)$, so $x \cdot h^{-1}(P) \in \hull(J)$ since $\hull(J)$ is a $G$-invariant subset of $\Prim C$. It is then clear 
	that $x \cdot P \in \hull(I)$, so $I$ is invariant.
\end{proof}

Now we proceed with the proof that \eqref{eq:sesB} is exact if and only if \eqref{eq:sesD} is. The key to the argument is verifying that the ideals $C_r^*(G, \B_I)$ and 
$C_r^*(G, \D_J)$ are paired under the Morita equivalence of $C_r^*(G, \B)$ and $C_r^*(G, \D)$. Indeed, it suffices to show that this Morita equivalence induces 
one between $C_r^*(G, \B_I)$ and $C_r^*(G, \D_J)$, and likewise between $C_r^*(G, \B^I)$ and $C_r^*(G, \D^J)$. The trick is to cut down 
the bundle $q : \E \to G$ that implements the $\B-\D$-equivalence in order to form a $\B_I - \D_J$-equivalence $\Eij \to G$. By taking fiberwise quotients of 
$\E$, we can also form a $\B^I - \D^J$-equivalence $\EIJ \to G$. We begin by defining
\[
	\Eij = \{ e \in \E : \la e, e \ra_\D \in D_J(s(e)) \} 
\]
where we have written $s(e)$ in place of $s(q(e))$. Notice that for each $x \in G$, 
\[
	E_{I,J}(x) = \{ e \in E(x) : \la e, e \ra_\D \in C_J(s(x)) \}.
\]
Since $E(x)$ is an $A(r(x))-C(s(x))$-imprimitivity bimodule, it should follow that the submodule $E_{I,J}(x)$ is an $I(r(x))-J(s(x))$-imprimitivity 
bimodule. Indeed, an argument along the lines of \cite[Lemma 3.1]{dana-marius} shows this to be the case.

\begin{lem}
\label{lem:fibermodules}
	For each $x \in G$, we have
	\[
		I(r(x)) \cdot E(x) = E_{I, J}(x) = E(x) \cdot J(s(x)).
	\]
	In other words, $E_{I,J}(x)$ is the closed submodule of $E(x)$ associated to $I(r(x))$ and $J(s(x))$ under the Rieffel correspondence. Thus $E_{I,J}(x)$ is an
	$I(r(x))-J(s(x))$-imprimitivity bimodule for each $x \in G$.
\end{lem}
\begin{proof}
	By definition (and Lemma 3.23 of \cite{TFB1}), we have $E_{I,J}(x) = E(x) \cdot J(s(x))$. Let $h_x^E = h_{r(x)} \circ h_x^C$ denote the Rieffel correspondence induced by 
	$E(x)$. It suffices to show that $I(r(x)) = h_x^E(J(s(x)))$. If $P \in \Prim A(r(x))$, then we have $I(r(x)) \subseteq P$ if and only if $I \subseteq (r(x), P)$, which in turn holds 
	if and only if $J \subseteq h^{-1}(r(x), P)$. However,
	\[
		h^{-1}(r(x), P) = (s(x), (h_x^C)^{-1}(h_{r(x)}^{-1}(P))) = (s(x), (h_x^E)^{-1}(P)),
	\]
	so $J \subseteq h^{-1}(r(x),P)$ if and only if $J(s(x)) \subseteq (h_x^E)^{-1}(P)$. It then follows that $\hull(I(r(x)) = h_x^E(\hull(J(s(x))))$, so $I(r(x)) = h_x^E(J(s(x)))$.
\end{proof}

Lemma \ref{lem:fibermodules} shows that the fibers of $\Eij$ are imprimitivity bimodules between the appropriate fibers of $\B_I$ and $\D_J$. With this result in
hand, we can proceed with the verification that $\Eij$ is a $\B_I-\D_J$-equivalence.

\begin{prop}
	The bundle $\qij : \Eij \to G$ is an upper semicontinuous Banach bundle, which is furthermore a $\B_I-\D_J$-equivalence. Consequently, $C^*(G,\B_I)$ and $C^*(G, \D_J)$ 
	are Morita equivalent, as are $C_r^*(G, \B_I)$ and $C_r^*(G, \D_J)$.
\end{prop}
\begin{proof}
	We equip the total space $\Eij$ with the topology inherited from $\E$. It is then necessary to check that the resulting bundle is upper semicontinuous, and that the 
	restriction of $q$ to $\Eij$ is an open map. Upper semicontinuity is fairly easy to verify. We just need to show that the set
	\[
		U_r = \{ e \in \Eij : \norm{e} < r \}
	\]
	is open for all $r \in \R$. Well,
	\[
		U_r = \{ e \in \E : \norm{e} < r \} \cap \Eij,
	\]
	and $\{ e \in \E : \norm{e} < r \}$ in $\E$ since $\E$ is an upper semicontinuous Banach bundle. Thus $U_r$ is open for all $r \in \R$, so $e \to \norm{e}$ 
	is upper semicontinuous from $\Eij$ to $\R$.
	
	The openness of $\qij$ is a little harder to verify, though the proof is very similar to that of \cite[Proposition 3.3]{dana-marius}. The argument relies upon 
	Lemma 1.15 of \cite{TFB2}. Let $e \in \Eij$ and put $x = \qij(e)$. Suppose $x_i \to x$ in $G$. Since $E_{I,J}(x) = E(x) \cdot J(s(x))$, we can write 
	$e = f \cdot a(s(x))$ for some $f \in E(x)$ and $a \in J$. Since the bundle map $q : \E \to G$ is open, we can pass to a subnet, relabel, and find elements 
	$f_i \in \E$ with $q(f_i) = x_i$ and $f_i \to f$. Since $a \in J = \Gamma_0(\go, \D_J)$ is continuous, $a(s(x_i)) \to a(s(x))$. Thus
	\[
		f_i \cdot a(s(x_i)) \to f \cdot a(s(x)) = e,
	\]
	since the action of $\D$ on $\E$ is continuous. Since $f_i \cdot a(s(x_i)) \in \Eij$ for all $i$, it follows that the restriction of $q$ to $\Eij$ is open. Therefore, 
	$\qij : \Eij \to G$ is an upper semicontinuous Banach bundle.

	Now we show that $\qij : \Eij \to G$ is a $\B_I - \D_J$-equivalence. Clearly the bundle $\qij : \Eij \to G$ should inherit natural actions and inner products from 
	$\E$, provided the restrictions of the maps defining those operations take values in the correct places. In particular, we first need to check that the actions 
	of $\B$ and $\D$ on $\E$ restrict to actions of $\B_I$ and $\D_J$, respectively, on $\Eij$. Let $e \in \Eij$ and $a \in \D_J$ with $s(e) = r(a)$, and $x = q(e)$ 
	and $y = p(a)$. Since $a \in D_I(y) = D(y) \cdot J(s(y))$, we can write $a = a' \cdot b$ for some $a' \in D(y)$ and $b \in J(s(x))$. Then
	\[
		e \cdot a'b = (e \cdot a') \cdot b \in E(xy) \cdot J(s(y)) = E(xy) \cdot J(s(xy)) \subseteq \Eij.
	\]
	Since we also know that $E_{I,J}(x) = I(r(x)) \cdot E(x)$ for all $x \in G$, the proof for the $\B_I$-action is similar.
	
	Now we check that the $\B$- and $\D$-valued sesquilinear forms on $\E$ restrict to forms ${_{\B_I} \la} \cdot, \cdot \ra : \Eij *_s \Eij \to \B_I$ and 
	$\la \cdot, \cdot \ra_{\D_J} : \Eij *_r \Eij \to \D_J$. Let $(e, f) \in \Eij *_r \Eij$, and write $f = f' \cdot a$ for some $f' \in E(q(f))$ and $a \in J(s(f))$. Then
	\[
		\la e, f \ra_{\D_J} = \la e, f \ra_\D = \la e, f' \cdot a \ra_\D = \la e, f' \ra_\D \cdot a \in J(s(f)) \subseteq \D_J.
	\]
	Again, the proof for the $\B_I$-valued form is similar. Also, it is clear that all the required axioms for the actions and inner products hold, since they hold in $\E$.
	
	Finally, we know from the previous lemma that $\Eij(x)$ is an $I(r(x))-J(s(x))$-imprimitivity bimodule for each $x \in G$. Therefore, all the axioms of 
	\cite[Definition 6.1]{muhly-williams} are satisfied, and $\qij : \Eij \to G$ is a $\B_I - \D_J$-equivalence.
\end{proof}

Now let $\mathsf{X}$ and $\mathsf{X}_r$ denote the $C^*(G,\B)-C^*(G,\D)$- and $C_r^*(G, \B)-C_r^*(G,\D)$-imprimitivity bimodules arising from $\E$. Likewise, 
we write $\mathsf{X}_J$ and $\mathsf{X}_{J,r}$ for the $C^*(G,\B_I)-C^*(G,\D_J)$- and $C_r^*(G, \B_I)-C_r^*(G,\D_J)$-imprimitivity bimodules afforded by $\Eij$.
If we refer back to the equivalence results of \cite{muhly-williams} and \cite{sims-williams2013}, we see that $\mathsf{X}_J$ and $\mathsf{X}_{J, r}$ both arise as 
completions of $\Gamma_c(G, \Eij)$ with respect to the norms induced from the full and reduced norms, respectively, on $\Gamma_c(G, \B_I)$ (or equivalently,
$\Gamma_c(G, \D_J)$). Observe also that $\Gamma_c(G, \Eij)$ embeds naturally into $\Gamma_c(G, \E)$, and it is not hard to see that this embedding is 
isometric for both norms. It follows that we have inclusions $\mathsf{X}_J \hookrightarrow \mathsf{X}$ and $\mathsf{X}_{J,r} \hookrightarrow \mathsf{X}_r$ of 
imprimitivity bimodules. Furthermore, it is easy to check that
\[
	C^*(G, \B_I) \cdot \mathsf{X} = \mathsf{X}_J = \mathsf{X} \cdot C^*(G, \D_J)
\]
and
\[
	C_r^*(G, \B_I) \cdot \mathsf{X}_r = \mathsf{X}_{J, r} = \mathsf{X}_r \cdot C_r^*(G, \D_J).
\]
Therefore, we have proven the following proposition.

\begin{prop}
\label{prop:rieffel}
	The ideals $C^*(G, \B_I) \subseteq C^*(G, \B)$ and $C^*(G, \D_J) \subseteq C^*(G, \D)$ are paired under the Rieffel correspondence induced from the 
	$C^*(G, \B)-C^*(G, \D)$-imprimitivity bimodule $\mathsf{X}$. The analogous statement holds for the reduced Fell bundle $C^*$-algebras.
\end{prop}

Now we turn our attention to the quotient Fell bundles $\B^I$ and $\D^J$. As in \cite{dana-marius}, for each $x \in G$ we define
\[
	E^{I,J}(x) = E(x)/E_{I,J}(x).
\]
Then $E^{I,J}(x)$ is automatically a $B^I(r(x))-D^J(s(x))$-imprimitivity bimodule by \cite[Proposition 3.25]{TFB1}. Now define 
\[
	\EIJ = \bigsqcup_{x \in G} E^{I,J}(x),
\]
and let $\qIJ : \EIJ \to G$ be the natural projection map. Our goal is to turn $\EIJ$ into a $\B^I-\D^J$-equivalence. We first need to equip $\EIJ$ with a topology that makes it into
an upper semicontinuous Banach bundle over $G$. We will follow the lead of \cite[Proposition 3.4]{dana-marius} and specify a collection of sections of $\EIJ$, and then use
them to generate a topology.

For each $x \in G$, let $\sigma_x : E(x) \to E^{I,J}(x)$ denote the quotient map. We can then define a bundle map $\hat{\sigma} : \E \to \EIJ$ by
\[
	\hat{\sigma}(e) = \sigma_{q(e)}(e).
\]
Given $f \in \Gamma_c(G, \E)$, define a section $\sigma(f) : G \to \EIJ$ by
\[
	\sigma(f)(x) = \hat{\sigma}(f(x)).
\]

\begin{prop}
\label{prop:ILT}
	The total space $\EIJ$ can be endowed with a topology making it into an upper semicontinuous Banach bundle such that
	\[
		\Gamma = \{ \sigma(f) : f \in \Gamma_c(G, \E) \}
	\]
	is dense in $\Gamma_c(G, \EIJ)$ with respect to the inductive limit topology.
\end{prop}
\begin{proof}
	The proof is nearly identical to that of \cite[Proposition 3.4]{dana-marius}. In light of the Hofmann-Fell theorem \cite[Theorem 1.2]{dana-marius}, it suffices 
	to prove the following:
	\begin{enumerate}
		\item For all $f \in \Gamma$, $x \mapsto \norm{f(x)}$ is upper semicontinuous from $G$ to $\R$.
		\item For all $x \in G$, $\{f(x) : f \in \Gamma\}$ is dense in $E^{I,J}(x)$.
	\end{enumerate}
	The second assertion follows immediately from the fact that $\E$ has enough sections and $\sigma_x : E(x) \to E^{I,J}(x)$ is surjective. Therefore, we will 
	focus on proving (1).
	
	Notice first that we can write
	\[
		\norm{\sigma(f)(x)}^2 = \norm{ {_{(A/I)(r(x))} \la} \sigma(f)(x), \sigma(f)(x) \ra} = \norm{\rho_{r(x)} \bigl( {_{A(r(x))} \la } f(x), f(x) \ra \bigr)},
	\]
	where $\rho_{r(x)} : A(r(x)) \to (A/I)(r(x)) \cong A(r(x))/I(r(x))$ is the quotient map. (The latter equality holds due to Proposition 3.25 of \cite{TFB1}.) If we let $\rho : A \to A/I$ 
	denote the quotient map, then $\rho$ is easily seen to be $C_0(\go)$-linear. Thus $\rho$ induces a continuous $C^*$-bundle homomorphism $\hat{\rho} : \A \to \A/\I$ 
	\cite[Proposition 3.4.16]{lalondethesis} whose restriction to $A(r(x))$ is $\rho_{r(x)}$. Thus the map
	\begin{equation}
	\label{eq:norm}
		x \mapsto \norm{\sigma(f)(x)}^2
	\end{equation}
	is the composition of the continuous map $G \to \A$ defined by
	\[
		x \mapsto {_{A(r(x))} \la} f(x), f(x) \ra
	\]
	with the continuous map $\A \to \A/\I$ given by $a \mapsto \hat{\rho}(a)$ and the upper semicontinuous map $a \mapsto \norm{a}$ from $\A/\I$ to $\R$. 
	Thus \eqref{eq:norm} is upper semicontinuous from $G$ to $\R$. It follows from \cite[Theorem 1.2]{dana-marius} that we can equip $\EIJ$ with a unique 
	topology making it into an upper semicontinuous Banach bundle with $\Gamma \subseteq \Gamma_0(G, \EIJ)$.
	
	Since $\Gamma$ is a $C_0(G)$-module, \cite[Lemma A.4]{muhly-williams} implies that $\Gamma$ is dense in $\Gamma_0(G, \EIJ)$. We claim that 
	$\Gamma$ is actually dense in $\Gamma_c(G, \EIJ)$ with respect to the inductive limit topology. To see why, let $g \in \Gamma_c(G, \EIJ) \subseteq 
	\Gamma_0(G, \EIJ)$ and find a net $f_i \in \Gamma_c(G, \E)$ such that $\sigma(f_i) \to g$ uniformly. Let $K = \supp(g)$, and choose $\varphi \in C_c(G)^+$ 
	such that $\varphi \vert_K \equiv 1$ and $\varphi(x) < 1$ for all $x \not\in K$. Put $g_i = \varphi \cdot \sigma(f_i) = \sigma(\varphi \cdot f_i)$. Notice that 
	$\varphi \cdot g = g$, so we have $g_i \to g$ uniformly. Moreover, $\supp(g_i) \subseteq \supp(\varphi)$ for all $i$, so $g_i \to g$ in the inductive limit 
	topology. Since $g_i \in \Gamma$ for all $i$, it follows that $\Gamma$ is dense in $\Gamma_c(G, \EIJ)$ with respect to the inductive limit topology.
\end{proof}

Before we can finishing proving that $\qIJ : \EIJ \to G$ is a $\B^I-\D^J$-equivalence, we need a quick lemma.

\begin{lem}
\label{lem:sigmacts}
	The bundle map $\hat{\sigma} : \E \to \EIJ$ is continuous.
\end{lem}
\begin{proof}
	The proof is nearly identical to the third paragraph of the proof of \cite[Proposition 3.4]{dana-marius}. Let $e \in \E$ and suppose $e_i \to e$. Put $x_i = q(e_i)$ and 
	$x = q(e)$, and choose $f \in \Gamma_c(G, \E)$ with $f(x) = e$. Then $f(x_i) \to e$, so $\norm{f(x_i) - e_i} \to 0$ by Lemma C.18 of \cite{TFB2}. Note that
	\[
		\norm{\hat{\sigma}(f(x_i)) - \hat{\sigma}(e_i)} = \norm{\hat{\sigma}(f(x_i) - e_i)} \leq \norm{f(x_i) - e_i} \to 0
	\]
	and that $\hat{\sigma}(f(x_i)) = \sigma(f)(x_i)$. Since $\sigma(f) \in \Gamma_c(G, \EIJ)$, we have $\sigma(f)(x_i) \to \sigma(f)(x) = \hat{\sigma}(e)$. Therefore, it follows 
	from \cite[Proposition C.20]{TFB2} that $\hat{\sigma}(e_i) \to \hat{\sigma}(e)$. Hence $\hat{\sigma}$ is continuous.
\end{proof}

\begin{prop}
	The bundle $\qIJ : \EIJ \to G$ is a $\B^I - \D^J$-equivalence.
\end{prop}
\begin{proof}
	First we need to produce commuting actions of $\B^I$ and $\D^J$ on the left and right, respectively, of $\EIJ$. Let $\hat{\rho}^\B : \B \to \B^I$ and 
	$\hat{\rho}^\D : \D \to \D^J$ denote the bundle maps coming from the canonical fiberwise quotient maps. We define
	\[
		\hat{\rho}^\B(b) \cdot \hat{\sigma}(e) = \hat{\sigma}(b \cdot e)
	\]
	for $b \in \B$ and $e \in \E$ with $s(b) = r(e)$ and
	\[
		\hat{\sigma}(e) \cdot \hat{\rho}^\D(d) = \hat{\sigma}(e \cdot d)
	\]
	for $e \in \E$ and $d \in \D$ with $s(e) = r(d)$. These formulas will clearly define commuting actions once we know that they are well-defined. We will prove 
	the left $\B^I$-action is well-defined, and the proof for the right $\D^J$ action is similar.
	
	In order to keep the notation manageable throughout the remainder of the proof, we will write $[ \cdot ]$ to denote the class of a bundle element in the appropriate 
	quotient bundle. That is, $[b] = \hat{\rho}^\B(b)$ for $b \in \B$, $[d] = \hat{\rho}^\D(d)$ for $d \in \D$, and $[e] = \hat{\sigma}(e)$ for $e \in \E$. Thus our actions 
	look like
	\[
		[b] \cdot [e] = [b \cdot e], \quad (b, e) \in \B * \E
	\]
	and
	\[
		[e] \cdot [d] = [e \cdot d], \quad (e, d) \in \E * \D.
	\]	
	Let $b \in \B$ and $e \in \E$, set $x = p(b)$ and $y = q(e)$, and suppose $s(x) = r(y)$. Let $b' \in B_I(x)$ and $e' \in E_{I,J}(y)$. Then
	\[
		[b+b'] \cdot [e+e'] = [(b+b') \cdot (e+e')] = [b \cdot e + b \cdot e' + b' \cdot e + b' \cdot e'].
	\]
	Observe that $b \cdot e' \in B(x) \cdot E_{I,J}(y)$, where
	\[
		B(x) \cdot E_{I,J}(y) = B(x) \cdot E(y) \cdot J(s(y)) = E(xy) \cdot J(s(xy)) = E_{I,J}(xy)
	\]
	by Lemma 6.2 of \cite{muhly-williams}. Similar arguments show that $B_I(x) \cdot E(y) = E_{I,J}(xy)$ and $B_I(x) \cdot E_{I,J}(y) = E_{I,J}(xy)$, so $b \cdot e', b' \cdot e$, 
	and $b' \cdot e'$ all belong to $E_{I,J}(xy)$. Therefore,
	\[
		(b+b') \cdot (e+e') \in b \cdot e + E_{I,J}(xy),
	\]
	so
	\[
		[(b+b') \cdot (e+e')] = [b \cdot e] = [b] \cdot [e],
	\]
	and the action is well-defined.
	
	We also need to check that the actions are continuous. Again, we do it for the $\B^I$-action, and the $\D^J$-action is similar. Suppose $[b_i] \to [b]$
	in $\B^I$ and $[e_i] \to [e]$ in $\EIJ$, where $s(b_i) = r(e_i)$ for all $i$ and $s(b) = r(e)$. Put $x_i = s(b_i)$, $x = s(b)$, $y_i = q(e_i)$, and $y = q(e)$.
	Choose sections $f \in \Gamma_c(G, \B)$ and $g \in \Gamma_c(G, \E)$ such that $\rho^\B(f)(x) =[f(x)] = [b]$ and $\sigma(g)(y) = [e]$, where 
	$\rho^\B : \Gamma_c(G, \B) \to \Gamma_c(G, \B^I)$ denotes the quotient map. Since $f(x_i) \to f(x)$ and $g(y_i) \to g(y)$, we have
	\[
		\rho^\B(f)(x_i) \cdot \sigma(g)(y_i) = [f(x_i) \cdot g(y_i)] \to [f(x) \cdot g(y)] = [b] \cdot [e]
	\]
	by Lemma \ref{lem:sigmacts}. Also,
	\begin{align*}
		& \norm{{\rho}^\B(f)(x_i) \cdot \sigma(g)(y_i) - [b_i] \cdot [e_i]} \\
			& \qquad \leq \norm{[f(x_i)] \cdot [g(y_i)] - [b_i] \cdot [g(y_i)]} + \norm{[b_i] \cdot [g(y_i)] - [b_i] \cdot [e_i]} \\
			& \qquad \leq \norm{[f(x_i)] - [b_i]} \norm{[g(y_i)]} + \norm{[g(y_i)] - [e_i]} \norm{[b_i]}.
	\end{align*}
	Since $[f(x_i)] \to [b]$ and $[b_i] \to [b]$, we know that $\norm{[f(x_i)] - [b_i]} \to 0$. Similarly, $\norm{[g(y_i)] - [e_i]} \to 0$. Moreover, $\norm{[g(y_i)]}$ and $\norm{[b_i]}$ 
	are eventually bounded (they converge and the norm is upper semicontinuous), so 
	\[
		\norm{{\rho}^\B(f)(x_i) \cdot \sigma(g)(y_i) - [b_i] \cdot [e_i]} = \norm{[f(x_i)] \cdot [g(y_i)] - [b_i] \cdot [e_i]} \to 0.
	\]
	It follows from Proposition C.20 of \cite{TFB2} that $[b_i] \cdot [e_i] \to [b] \cdot [e]$. Therefore, the $\B^I$-action is 
	continuous.
	
	Now we need to define sesquilinear forms on $\EIJ$. For $e, f \in \E$ with $s(e) = s(f)$, we define
	\[
		{_{\B^I} \la } [e], [f] \ra = [ {_{\B} \la} e, f \ra ].
	\]
	Similarly, if $r(e) = r(f)$, we set
	\[
		\la [e], [f] \ra_{\D^J} = [ \la e, f \ra_\D ].
	\]
	Again, we need to check that these forms are well-defined. Let $e, f \in \E$ with $r(e) = r(f)$, and put $x = q(e)$ and $y = q(f)$. If $e' \in E_{I,J}(x)$ and 
	$f' \in E_{I,J}(y)$, then
	\[
		{_\B \la} e+e', f+f' \ra = {_\B \la} e, f \ra + {_\B \la} e', f \ra + {_\B \la} e, f' \ra + {_\B \la} e', f' \ra.
	\]
	Since $E_{I,J}(x) = I(r(x)) \cdot E(x)$, we can write $e' = a \cdot e''$ for some $a \in I(r(x))$ and $e'' \in E(x)$. Thus
	\[
		{_\B \la} e', f \ra = {_\B \la} a \cdot e'', f \ra = a {_\B \la} e'', f \ra \in I(r(x)) \cdot B(xy^{-1}) = B_{I,J}(xy^{-1}).
	\]
	Similarly, $f' = b \cdot f''$ for some $b \in I(r(y))$ and $f'' \in E(y)$, so
	\[
		{_\B \la} e, f' \ra = {_\B \la} e, b \cdot f'' \ra = b^* {_\B \la} e, f'' \ra \in I(r(x)) \cdot B(xy^{-1}) = B_{I,J}(xy^{-1}).
	\]
	Finally, ${_\B \la} e', f' \ra \in B_{I,J}(xy^{-1})$ since $e', f' \in \Eij$, so
	\[
		{_\B \la} e + e', f+f' \ra \in {_\B \la} e, f \ra + B_{I,J}(xy^{-1}).
	\]
	Thus $[ {_\B \la} e+e', f+f' \ra] = [ {_\B \la} e, f \ra]$ in $\B^I$, and the form is well-defined. The proof for the $\D^J$-valued sesquilinear form is similar.
	
	Now we show the forms are continuous. Suppose $[e_i] \to [e]$ and $[f_i] \to [f]$ in $\EIJ$, where $s(e_i) = s(f_i)$ for all $i$ and $s(e) = s(f)$. Let $x_i = q(e_i)$, $x = q(e)$,
	$y_i = q(f_i)$, and $y = q(f)$, and choose sections $\xi, \eta \in \Gamma_c(G, \E)$ such that $[\xi(x)] = [e]$ and $[\eta(y)] = f$. Then $\xi(x_i) \to \xi(x)$ and 
	$\eta(y_i) \to \eta(y)$, so
	\[
		{_\B \la} \xi(x_i), \eta(y_i) \ra \to {_\B \la} \xi(x), \xi(y) \ra.
	\]
	Hence
	\[
		[{_\B \la} \xi(x_i), \eta(y_i) \ra] \to [{_\B \la} \xi(x), \xi(y) \ra]
	\]
	or equivalently,
	\[
		{_{\B^I} \la} [\xi(x_i)], [\eta(y_i)] \ra \to {_{\B^I} \la } [\xi(x)], [\eta(y)] \ra = {_{\B^I} \la} [e], [f] \ra.
	\]
	Moreover,
	\begin{align*}
		&\norm{ {_{\B^I} \la} [\xi(x_i)], [\eta(y_i)] \ra - {_{\B^I} \la} [e_i], [f_i] \ra} \\
			& \quad \leq \norm{ {_{\B^I} \la} [\xi(x_i)], [\eta(y_i)] \ra - {_{\B^I} \la} [\xi(x_i)], [f_i] \ra} + \norm{ {_{\B^I} \la} [\xi(x_i)], [f_i] \ra  - {_{\B^I} \la} [e_i], [f_i] \ra} \\
			& \quad = \norm{ {_{\B^I} \la} [\xi(x_i)], [\eta(y_i)] - [f_i] \ra } + \norm{ {_{\B^I} \la} [\xi(x_i)] - [e_i] , [f_i] \ra } \\
			& \quad \leq \norm{ [\xi(x_i)] } \norm{ [\eta(y_i)] - [f_i] } + \norm{ [\xi(x_i)] - [e_i] } \norm{ [f_i] },
	\end{align*}
	which tends to 0 by the same reasoning as that for the module actions. Therefore, ${_{\B^I} \la} [e_i], [f_i] \ra \to {_{\B^I} \la} [e], [f] \ra$ by Proposition C.20 
	of \cite{TFB2}.
	
	We also have some algebraic conditions to verify, which are fairly straightforward. First notice that if $(e, f) \in \E *_s \E$, 
	\[
		{_{\B^I} \la} [e], [f] \ra^* = [{_\B \la} e, f \ra^*] = [{_\B \la} f, e \ra] = {_{\B^I} \la} [f], [e] \ra.
	\]
	Furthermore, if $b \in \B$ with $s(b) = r(e)$, then
	\[
		[b] \cdot {_{\B^I} \la} [e], [f] \ra = [ b \cdot {_\B \la} e, f \ra] = [ {_\B \la} b \cdot e, f \ra] = {_{\B^I} \la} [b] \cdot [e], [f] \ra,
	\]
	and similarly for the right $\D^J$-action. Also, if $g \in \E$ with $r(g) = r(f)$, then
	\[
		{_{\B^I} \la} [e], [f] \ra \cdot [g] = [ {_\B \la} e, f \ra \cdot g] = [ e \cdot \la f, g \ra_\D] = [e] \cdot \la [e], [f] \ra_{\D^J},
	\]
	as required.
	
	Finally, we have already argued that $E^{I,J}(x)$ is a $A(r(x))-C(s(x))$-imprimitivity bimodule for all $x \in G$. Therefore, $\EIJ$ is a $\B^I-\D^J$-equivalence.
\end{proof}

Since $\EIJ$ is a $\B^I-\D^J$-equivalence, we know that $\Gamma_c(G, \EIJ)$ completes to a $C_r^*(G, \B^I) - C_r^*(G, \D^J)$-imprimitivity bimodule 
$\mathsf{X}^{I,J}$. We would like to know that this module is compatible with the $C_r^*(G, \B) - C_r^*(G, \D)$-imprimitivity bimodule $\mathsf{X} = 
\overline{\Gamma_c(G, \E)}$ in a certain sense. We claim that the continuous bundle map $\hat{\sigma} : \E \to \Eij$ induces a linear map 
$\sigma : \mathsf{X} \to \mathsf{X}^{I,J}$, which is characterized by
\[
	\sigma(f)(x) = \hat{\sigma}(f(x))
\]
for $f \in \Gamma_c(G, \E)$. Moreover, $\sigma$ respects the module actions on $\mathsf{X}$ and $\mathsf{X}^{I,J}$.

\begin{prop}
\label{prop:sigma}
	The map $\sigma : \Gamma_c(G, \E) \to \Gamma_c(G, \EIJ)$ defined above extends to a surjective linear map $\sigma : \mathsf{X} \to \mathsf{X}^{I,J}$. Moreover, for 
	all $x,y \in \mathsf{X}$, $a \in C_r^*(G, \B^I)$, and $b \in C_r^*(G, \D^J)$, we have
	\[
		\sigma(a \cdot x) = \rho^\B(a) \cdot \sigma(x), \quad \sigma(x \cdot b) = \sigma(x) \cdot \rho^\D(b)
	\]
	and
	\[
		\rho^\B( {_*\la} x, y \ra ) = {_* \la} \sigma(x), \sigma(y) \ra, \quad \rho^\D( \la x, y \ra_*) = \la \sigma(x), \sigma(y) \ra_*,
	\]
	where $\rho^\B : C_r^*(G, \B) \to C_r^*(G, \B^I)$ and $\rho^\D : C_r^*(G, \D) \to C_r^*(G, \D^J)$ denote the quotient maps.
\end{prop}
\begin{proof}
	The topology on $\EIJ$ was defined in such a way to ensure that 
	\[
		\Gamma = \{ \sigma(f) : f \in \Gamma_c(G, \E) \}
	\]
	is contained in $\Gamma_c(G, \EIJ)$. Therefore, $\sigma$ defines a map from $\Gamma_c(G, \E)$ into $\Gamma_c(G, \EIJ)$, which is easily seen to
	be linear. If we let $f \in \Gamma_c(G, \B)$ and $\xi \in \Gamma_c(G, \E)$, then
	\begin{align*}
		\sigma(f \cdot \xi)(x) &= \hat{\sigma}(f \cdot \xi (x)) \\
			&= \hat{\sigma} \left( \int_G f(y) \cdot \xi(y^{-1}x) \, d\lambda^{r(x)}(y) \right) \\
			&= \int_G \sigma_x \bigl( f(y) \cdot \xi(y^{-1}x) \bigr) \, d\lambda^{r(x)}(y) \\
			&= \int_G \rho^\B_y \bigl( f(y) \bigr) \cdot \sigma_{y^{-1}x} \bigl( \xi(y^{-1}x) \bigr) \, d\lambda^{r(x)}(y) \\
			&= \bigl( \rho^B(f) \cdot \sigma(\xi) \bigr) (x)
	\end{align*}
	A similar computation shows that $\sigma(\xi \cdot g) = \sigma(\xi) \cdot \rho^\D(g)$ for all $\xi \in \Gamma_c(G, \E)$ and $g \in \Gamma_c(G, \D)$. Now suppose 
	$\xi, \eta \in \Gamma_c(G, \E)$. Then
	\begin{align*}
		\rho^\B \bigl( {_* \la} \xi, \eta \ra \bigr) &= \rho^\B_x \bigl( {_* \la} \xi, \eta \ra(x) \bigr) \\
			&= \rho^\B_x \left( \int_G {_\B \la} \xi(xy), \eta(y) \ra \, d\lambda^{s(x)}(y) \right) \\
			&= \int_G \rho^\B_x \bigl( {_\B \la} \xi(xy), \eta(y) \ra \bigr) \, d\lambda^{s(x)}(y) \\
			&= \int_G {_{\B^I} \la} \sigma_{xy} \bigl(\xi(xy)\bigr), \sigma_y \bigl( \eta(y) \bigr) \ra \, d\lambda^{s(x)}(y) \\
			&= {_* \la} \sigma(\xi), \sigma(\eta) \ra(x).
	\end{align*}
	Similarly, $\rho^\D(\la \xi, \eta \ra_*) = \la \sigma(\xi), \sigma(\eta) \ra_*$ for all $(\xi, \eta) \in \E *_r \E$.
	Using the latter fact, it is fairly easy to show that $\sigma$ is bounded:
	\[
		\norm{\sigma(f)}_* = \norm{\la \sigma(f), \sigma(f) \ra_*}^{1/2} = \norm{\rho^\D \bigl( \la f, f \ra_* \bigr)}^{1/2} \leq \norm{ \la f, f \ra_*}^{1/2} = \norm{f}.
	\]
	Thus $\sigma$ extends to a module map $\sigma : \mathsf{X} \to \mathsf{X}^{I,J}$.
	
	All that remains is to see that $\sigma$ is surjective. We already know that the range of $\sigma$ is dense in $\Gamma_c(G, \EIJ)$ with respect to the 
	inductive limit topology, so we just need to show that density in norm follows. Suppose $\xi_i \to \xi$ in $\Gamma_c(G, \EIJ)$ in the inductive limit topology. 
	Then arguments like those of \cite[Lemma 8.1(b)]{mw08} and \cite[Lemma 5.5]{lalonde2014} show that $\la \xi_i - \xi, \xi_i - \xi \ra_* \to 0$ in 
	$\Gamma_c(G, \D^J)$ with respect to the inductive limit topology. It is straightforward to show that $\la \xi_i - \xi, \xi_i - \xi \ra_* \to 0$ uniformly: observe that
	\[
		\norm{ \la \xi_i - \xi, \xi_i - \xi \ra_* }_\infty = \sup_{x \in G} \norm{ \la \xi_i - \xi, \xi_i - \xi \ra_*(x) },
	\]
	where
	\begin{align*}
		\norm{ \la \xi_i - \xi, \xi_i - \xi \ra_*(x) } &= \norm{ \int_G \la (\xi_i - \xi)(y^{-1}), (\xi_i - \xi)(y^{-1}x) \ra_* \, d\lambda^{r(x)}(y)} \\
			& \leq \int_G \norm{ \la (\xi_i - \xi)(y^{-1}), (\xi_i - \xi)(y^{-1}x) \ra_*} \, d\lambda^{r(x)}(y) \\
			& \leq \int_G \norm{ (\xi_i - \xi)(y^{-1}) } \norm{ (\xi_i - \xi)(y^{-1}x) } \, d\lambda^{r(x)}(y).
	\end{align*}
	For sufficiently large $i$, the sets $\supp(\xi_i - \xi)$ are contained in a fixed compact set $K$, so we eventually have
	\[
		\norm{ \la \xi_i - \xi, \xi_i - \xi \ra_*(x) } \leq \norm{\xi_i - \xi}_\infty^2 \cdot \lambda^{r(x)}(K).
	\]
	Thus
	\[
		\norm{ \la \xi_i - \xi, \xi_i - \xi \ra_* }_\infty \leq \sup_{x \in G} \norm{\xi_i - \xi}_\infty^2 \cdot \lambda^{r(x)}(K) = \norm{\xi_i - \xi}_\infty^2 \cdot 
			\sup_{u \in \go} \lambda^u(K).
	\]
	Since $K$ is compact, the supremum is finite. Thus $\xi_i \to \xi$ uniformly implies that $\la \xi_i - \xi, \xi_i - \xi \ra_* \to 0$ uniformly in $\Gamma_c(G, \D^J)$. 
	It remains to see that the functions $\la \xi_i - \xi, \xi_i - \xi \ra_*$ are eventually supported in a fixed compact set. Since $\xi_i \to \xi$ in the inductive limit 
	topology on $\Gamma_c(G, \EIJ)$, there is a compact set $K_0 \subset G$ that eventually contains $\supp(\xi_i)$ and $\supp(\xi)$. Form the compact set 
	$K_0 *_r K_0 \subseteq G *_r G$, and let $\varphi : G *_r G \to G$ be the map defined by $\varphi(z, w) = z^{-1} w$. Notice that $K = \varphi(K_0 *_r K_0)$ is 
	compact. Furthermore,
	\[
		\la \xi_i - \xi, \xi_i - \xi \ra_*(x) =  \int_G \la (\xi_i - \xi)(y^{-1}), (\xi_i - \xi)(y^{-1}x) \ra_* \, d\lambda^{r(x)}(y),
	\]
	and the integrand is nonzero only when $y^{-1}, y^{-1} x \in K_0$. Suppose $y^{-1} x = z \in K_0$, or $x = yz$. Then $x \in K$. Therefore, $\la \xi_i - \xi, \xi_i - \xi \ra_*(x)$
	is zero whenever $x \not\in K$, and it follows that $\supp(\la \xi_i - \xi, \xi_i - \xi \ra_*)$ is eventually contained in $K$. Thus $\la \xi_i - \xi, \xi_i - \xi \ra_* \to 0$ in the inductive
	limit topology. It is then straightforward to see that
	\[
		\norm{\xi_i - \xi}^2 = \norm{\la \xi_i - \xi, \xi_i - \xi \ra}_r \leq \norm{\la \xi_i - \xi, \xi_i - \xi \ra} \to 0,
	\]
	so $\xi_i \to \xi$ with respect to the norm on $\mathsf{X}^{I,J}$. Thus density with respect to the inductive limit topology implies norm density, so the range of $\sigma$ is 
	dense in $\Gamma_c(G, \EIJ)$. Thus $\sigma : \mathsf{X} \to \mathsf{X}^{I,J}$ is surjective.
\end{proof}

With the last proposition in hand, we are almost ready to prove our main result. There is one lemma regarding Morita equivalence that we need first, however. It 
is likely evident to experts, but we present a complete proof here. One can think of it as a partial converse to Proposition 3.25 of \cite{TFB1}.

\begin{lem}
\label{lem:MElemma}
	Let $A$ and $B$ be $C^*$-algebras and suppose $I \subseteq A$ and $J \subseteq B$ are ideals. Suppose further that $\mathsf{X}$ is an $A-B$-imprimitivity 
	bimodule, $\mathsf{Y}$ is an $A/I-B/J$-imprimitivity bimodule, and there is a surjective linear map $q : \mathsf{X} \to \mathsf{Y}$ satisfying
	\begin{align*}
		q(a \cdot x) &= p^A(a) \cdot q(x) \\
		q(x \cdot b) &= q(x) \cdot p^B(b) \\
		p^A \bigl( {_A \la} x, y \ra \bigr) &= {_{A/I} \la} q(x), q(y) \ra \\
		p^B \bigl( \la x, y \ra_B \bigr) &= \la q(x), q(y) \ra_{B/J}
	\end{align*}
	for all $x, y \in \mathsf{X}$, $a \in A$, and $b \in B$, where $p^A : A \to A/I$ and $p^B : B \to B/J$ denote the canonical quotient maps. Then $I$ and $J$ are 
	paired under the Rieffel correspondence associated to $\mathsf{X}$.
\end{lem}
\begin{proof}
	Let $\pi : B/J \to B(\Hil)$ be a faithful representation, and put $\tilde{\pi} = \pi \circ p^B$. Then $\tilde{\pi}$ is a representation of $B$ on $\Hil$ with kernel $J$. 
	Form the induced representation $\rho = \mathsf{Y}\text{--}\Ind \pi$ of $A/I$ on $Y \otimes \Hil$, and note that $\rho$ is faithful. Thus $\tilde{\rho} = \rho \circ p^A$ 
	is a representation of $A$ on $\mathsf{Y} \otimes \Hil$ with kernel $I$. It will therefore suffice to show that $\tilde{\rho}$ is unitarily equivalent to 
	$\mathsf{X}\text{--}\Ind \tilde{\pi}$, which acts on $\mathsf{X} \otimes \Hil$.
	
	Define $U_0 : \mathsf{X} \odot \Hil \to \mathsf{Y} \odot \Hil$ on elementary tensors by $U_0(x \otimes h) = q(x) \otimes h$. Observe that
	\begin{align*}
		\ip{U_0(x \otimes h)}{U_0(y \otimes k)} &= \ip{q(x) \otimes h}{q(y) \otimes k} \\
			&= \ip{\pi \bigl( \la q(y), q(x) \ra_{B/I} \bigr) h}{k} \\
			&= \ip{\pi \bigl( p^B(\la y, x \ra_B) \bigr) h }{k} \\
			&= \ip{ \tilde{\pi} \bigl( \la y, x \ra_B \bigr) h}{k} \\
			&= \ip{x \otimes h}{y \otimes k},
	\end{align*}
	where the last inner product is taken in $X \otimes \Hil$. Thus $U_0$ is isometric. It maps $X \odot \Hil$ onto $Y \odot \Hil$ since $q$ is surjective. 
	Thus $U_0$ extends to a unitary $U : X \otimes \Hil \to Y \otimes \Hil$.
	
	We now claim that $U$ intertwines $\tilde{\rho}$ and $\mathsf{X}\text{--}\Ind \tilde{\pi}$. If $a \in A$ and $x \otimes h \in X \otimes \Hil$, then
	\begin{align*}
		\tilde{\rho}(a) U(x \otimes h) &= \tilde{\rho}(a) (q(x) \otimes h) \\
			&= \rho(p^A(a))(q(x) \otimes h) \\
			&= p^A(a) \cdot q(x) \otimes h \\
			&= q(a \cdot x) \otimes h \\
			&= U(a \cdot x \otimes h) \\
			&= U \cdot \mathsf{X}\text{--}\Ind \tilde{\pi}(a) (x \otimes h).
	\end{align*}
	Thus $\tilde{\rho}(a) U = U (\mathsf{X}\text{--}\Ind \tilde{\pi}(a))$ for all $a \in A$. Consequently, $\ker (\mathsf{X}\text{--}\Ind \tilde{\pi}) = \ker \tilde{\rho} = I$. 
	Since $\ker \tilde{\pi} = J$, we can conclude that $I = \mathsf{X}\text{--}\Ind J$.
\end{proof}

\begin{thm}
\label{thm:exactsequences}
	Let $p_\B : \B \to G$ and $p_\D : \D \to G$ be Fell bundles over a groupoid $G$, and let $q : \E \to G$ be a $\B - \D$-equivalence over the trivial $G-G$-equivalence. 
	Suppose $J \subseteq C = \Gamma_0(\go, \D)$ is a $G$-invariant ideal, and let $I$ be the corresponding ideal in $\Gamma_0(\go, \B)$. Then the sequence
	\[
		0 \to C_r^*(G, \B_I) \to C_r^*(G, \B) \to C_r^*(G, \B^I) \to 0
	\]
	is exact if and only if
	\[
		0 \to C_r^*(G, \D_J) \to C_r^*(G, \D) \to C_r^*(G, \D^J) \to 0
	\]
	is exact.
\end{thm}
\begin{proof}
	Let $\mathsf{X}$ denote the $C_r^*(G, \B)-C_r^*(G, \D)$-imprimitivity bimodule arising from $\E$. We have shown that there is a $\B_I - \D_J$-equivalence 
	$\qij : \Eij \to G$, and that the resulting $C_r^*(G, \B_I)-C_r^*(G, \D_J)$-imprimitivity bimodule $\mathsf{X}_{I,J}$ embeds naturally into $\mathsf{X}$. Thus 
	the Rieffel correspondence associated to $\mathsf{X}$ pairs the ideals $C_r^*(G, \B_I)$ and $C_r^*(G, \D_J)$. Likewise, we have a $\B^I - \D^J$-equivalence 
	$\qIJ : \EIJ \to G$, which yields a $C_r^*(G, \B^I) - C_r^*(G, \D^J)$-imprimitivity bimodule $\mathsf{X}^{I,J}$. Furthermore, the map 
	$\sigma : \mathsf{X} \to \mathsf{X}^{I,J}$ satisfies all the conditions of Lemma \ref{lem:MElemma}, so the kernels of the quotient maps $C_r^*(G, \B) \to C_r^*(G, \B^I)$ 
	and $C_r^*(G, \D) \to C_r^*(G, \D^J)$ are matched up by the Rieffel correspondence. Hence the sequence associated to $\B$ is exact if and only if the sequence
	arising from $\D$ is.
\end{proof}

The main application that we have had in mind all along is the following.

\begin{thm}
\label{thm:fellexact}
	Let $G$ be a second countable locally compact Hausdorff groupoid, and let $p : \B \to G$ be a Fell bundle over $G$. If $G$ is exact, then given any invariant ideal 
	$I \subseteq A = \Gamma_0(G, \B)$, the sequence
	\[
		0 \to C_r^*(G, \B_I) \to C_r^*(G, \B) \to C_r^*(G, \B^I) \to 0
	\]
	is exact.
\end{thm}
\begin{proof}
	Let $(\K(\V), G, \alpha)$ be the groupoid dynamical system associated to $\B$ under the stabilization theorem of \cite{IKSW}. Then the Fell bundle arising from 
	$(\K(\V), G, \alpha)$ is equivalent to $\B$. Let $J \subseteq \K(V)$ be the ideal corresponding to $I$. Then our main theorem tells us that
	\[
		0 \to C_r^*(G, \B_I) \to C_r^*(G, \B) \to C_r^*(G, \B^I) \to 0
	\]
	is exact if and only if
	\[
		0 \to \mathcal{J} \rtimes_{\alpha_J, r} G \to \K(\V) \rtimes_{\alpha, r} G \to \K(\V)/\mathcal{J} \rtimes_{\alpha^J, r} G \to 0
	\]
	is exact. But the latter sequence is exact since $G$ is an exact groupoid.
\end{proof}

By specializing to ideals coming from open invariant subsets of $\go$, we have the following analogue of \cite[Lemma 9]{aidan-dana} as a special case of Theorem 
\ref{thm:fellexact}.

\begin{cor}
\label{cor:fellinnerexact}
	Suppose $G$ is an exact groupoid, and let $p: \B \to G$ be a separable Fell bundle over $G$. Suppose $U \subseteq \go$ is open and invariant, and put
	$F = \go \backslash U$. Then the sequence
	\[
		0 \to C_r^*(G\vert_U, \B) \to C_r^*(G, \B) \to C_r^*(G\vert_F, \B) \to 0
	\]
	is exact.
\end{cor}
\begin{proof}
	Let $A = \Gamma_0(\go, \B)$ be the unit $C^*$-algebra of $\B$. It is observed in \cite{aidan-dana} that
	\[
		I = \{ a \in A : a(u) = 0 \text{ for all } u \in F \}
	\]
	is a $G$-invariant ideal of $A$. Moreover, the associated Fell bundle $\B_I$ can be identified with $\B \vert_{G \vert_U}$, and likewise for $\B^I$ and $\B \vert_{G \vert_F}$.
	Since $G$ is exact, the result now follows from Theorem \ref{thm:fellexact}.
\end{proof}

We also obtain a reduced version of \cite[Corollary 10]{aidan-dana}, which will be useful in the next section.

\begin{cor}
\label{cor:orbitfibers}
	Suppose $G$ is an exact groupoid, and let $p: \B \to G$ be a separable Fell bundle over $G$. If the orbit space $G \backslash \go$ is Hausdorff, then $C_r^*(G, \B)$
	is a $C_0(G \backslash \go)$-algebra with fibers
	\[
		C_r^*(G, \B)_{[u]} = C_r^*(G \vert_{[u]}, \B)
	\]
	for each $u \in \go$.
\end{cor}
\begin{proof}
	Notice first that $C_r^*(G, \B)$ is a quotient of $C^*(G, \B)$, which is a $C_0(G \backslash \go)$-algebra by \cite[Corollary 10]{aidan-dana}. Hence $C_r^*(G, \B)$
	is itself a $C_0(G \backslash \go)$-algebra by \cite[Lemma 1.3]{dana-marius}.

	Now let $u \in \go$. Since $G \backslash \go$ is Hausdorff, the orbit $[u]$ is closed, so $U = \go \backslash [u]$ is open and invariant. The fiber $C_r^*(G, \B)_{[u]}$ 
	is the quotient of $C_r^*(G, \B)$ by the ideal
	\[
		J_{[u]} = \overline{\spa} \bigl\{ \varphi \cdot a : \varphi \in C_0(G \backslash \go), \, \varphi([u]) = 0, \text{ and } a \in C_r^*(G, \B) \bigr\},
	\]
	which we can identify with $C_r^*(G \vert_U, \B)$. We then have $C_r^*(G, \B)/J_{[u]} = C_r^*(G \vert_{[u]}, \B)$ by Corollary \ref{cor:fellinnerexact}, since $G$
	is exact.
\end{proof}

\begin{rem}
	Note that in order to use the stabilization theorem, it was necessary only to work with equivalent Fell bundles over a common groupoid $G$. It is likely that 
	one can extend Theorem \ref{thm:exactsequences} to a result for equivalent Fell bundles $\B \to G$ and $\D \to H$ over \emph{different} groupoids (indeed, 
	Theorem 3.3 of \cite{lalondeexact} is a special case), though one likely needs to consider a more refined notion of Morita equivalence. The proof we have 
	given here breaks down at the very first step in general, as illustrated by a fairly simple example.
	
	Let $G$ and $H$ be groupoids, and suppose $Z$ is a $G-H$-equivalence. Let $\B = G \times \C$ and $\D = H \times \C$ denote the trivial line
	bundles over $G$ and $H$, respectively. Then $C^*(G, \B) = C^*(G)$ and $C^*(H, \D) = C^*(H)$, and likewise for the reduced algebras.
	If we let $\E = Z \times \C$, then $\E$ Is a $\B-\D$-equivalence (see Example 5.10 of \cite{mw08}, for example). However, notice that
	$A = C_0(\go)$ and $C = C_0(\ho)$, which are Morita equivalent if and only if $\go$ and $\ho$ are homeomorphic. Thus the unit $C^*$-algebras need not be Morita 
	equivalent in general.
\end{rem}

\section{An Application to Groupoid Extensions} 
\label{sec:extensions}
As an application of our main result from the last section, we show that any extension of an exact groupoid by an exact groupoid is again exact. Aside from being
interesting in its own right, this theorem provides a significant strengthening of the exactness results for twisted crossed products in Section \ref{sec:nuclear}. The 
proof requires us to first adapt a recent construction involving iterated Fell bundle $C^*$-algebras, due to Buss and Meyer, to the reduced setting. In this sense, 
our argument is in the same spirit as the original proof for groups by Kirchberg and Wassermann, which involved some delicate manipulation of iterated twisted 
reduced crossed products. It is worth noting that there are already some partial results in this direction, namely \cite[Theorem 3.4]{DKR} (for certain Fell bundles 
over \'{e}tale groupoids) and \cite[Theorem 3.8]{lalonde2017} (for dynamical systems associated to transformation groupoids). 

Suppose we have an extension of locally compact Hausdorff groupoids:
\[
	\xymatrix{\go \ar[r] & S \ar[r]^i & E \ar[r]^j & G \ar[r] & \go}.
\]
Here we require that $i$ is a homeomorphism onto a closed subgroupoid of $E$ and $j$ is a continuous open surjection. It is implicit that $S^{{(0)}} = E^{{(0)}} = \go$ 
and that $i(S) \subseteq \iso(E)$, so $S$ is necessarily a group bundle. (Unlike the extensions in Section \ref{sec:nuclear}, we do not assume the groups are abelian.) 
We also assume that $S$ and $G$ are equipped with Haar systems $\{\mu^u\}_{u \in S^{(0)}}$ and $\{\lambda^u\}_{u \in \go}$, respectively. It follows from 
\cite[Theorem 5.1]{BM} that $E$ can be endowed with a Haar system $\{\nu^u\}_{u \in E^{(0)}}$ characterized by
\begin{equation}
\label{eq:haar}
	\int_E f(e) \, d\nu^u(e) = \int_G \int_S f(e' g) \, d\mu^{s(e')}(g) \, d\lambda^u(j(e'))
\end{equation}
for $f \in C_c(E)$, where $e' \in E$ is any element satisfying $r(e') = u$. We will always assume that $E$ is equipped with this Haar system. Note that \eqref{eq:haar}
is a direct generalization of the natural Haar system on a twist, as defined in \cite{mw95}.

Given a Fell bundle $p : \B \to E$, Buss and Meyer \cite{BM} showed how to decompose $C^*(E, \B)$ as an ``iterated crossed product'' by producing a Fell bundle $\Cgh$ 
over $G$ with $\Gamma_0(\go, \Cgh) = C^*(S, \B \vert_S)$ and $C^*(G, \Cgh) \cong C^*(E, \B)$. This result can be thought of as a far-reaching generalization of classical 
theorems (such as \cite[Proposition 7.28]{TFB2}) for decomposing crossed products by groups into iterated twisted crossed products. We first present an outline of this 
construction, and then we show how it can be adapted (under certain circumstances) to reduced Fell bundle $C^*$-algebras. Notice first that $C^*(S, \B \vert_S)$ is a 
$C_0(\go)$-algebra with fibers
\[
	C^*(S, \B \vert_S)_u = C^*(S_u, \B \vert_u)
\]
for all $u \in \go$, by Corollary 10 of \cite{aidan-dana}. Thus we set $\Cgh_u = C^*(S_u, \B \vert_u)$ for each $u \in \go$. Next, for each $x \in G$ the set
$E_x = j^{-1}(x)$ is a $S_{r(x)}-S_{s(x)}$-equivalence, and $\B \vert_{E_x}$ implements an equivalence between the Fell bundles $\B \vert_{S_{r(x)}}$ and 
$\B \vert_{S_{s(x)}}$. Consequently $\Gamma_c(E_x, \B \vert_{E_x})$ completes to a $C^*(S_{r(x)}, \B \vert_{S_{r(x)}}) - C^*(S_{s(x)}, \B \vert_{S_{s(x)}})$-imprimitivity 
bimodule, which we call $\Cgh_x$. Now given $f \in \Gamma_c(E_x, \B \vert_{E_x})$ and $g \in \Gamma_c(E_y, \B \vert_{E_y})$ with $(x, y) \in \gtwo$, we 
define $f*g \in \Gamma_c(E_{x y}, \B \vert_{E_{x y}})$ by
\[
	f*g(e) = \int_G f(e' y) d\lambda^{s(e')}(y)
\]
for any $e' \in E$ with $j(e') = x$. Moreover, this multiplication map is bilinear and extends to the completions $\Cgh_x$ and $\Cgh_y$. We can also define an involution 
as follows: given $f \in \Gamma_c(E_x, \B \vert_{E_x})$, define $f^* \in \Gamma_c(E_{x^{-1}}, \B \vert_{E_{x^{-1}}})$ by
\[
	f^*(e) = \bigl( f(e^{-1}) \bigr)^*.
\]
Again, this definition extends to the completion $\Cgh_x$. Finally, $\Cgh$ can be equipped with a topology that makes it into the total space of a Fell bundle
over $G$. The map $\xi \mapsto \tilde{\xi}$ from $\Gamma_c(E, \B) \to \Gamma_c(G, \Cgh)$, where
\[
	\tilde{\xi}(x) = \xi \vert_{E_x},
\]
extends to an isomorphism of $C^*(E, \B)$ onto $C^*(G, \Cgh)$ \cite[Theorem 6.2]{BM}.

Now we set about producing a reduced version of the Buss-Meyer construction. We will attempt to do so in a way that allows us to avoid rechecking all the details in the 
reduced setting, and whose proofs are in the same spirit as some of the others in this paper. 

\begin{prop}
	The reduced Fell bundle $C^*$-algebra $C^*_r(S, \B \vert_S)$ is a $C_0(\go)$-algebra. Moreover, if $S$ is exact then
	\[
		C_r^*(S, \B \vert_S)_u = C_r^*(S_u, \B \vert_{S_u})
	\]
	for all $u \in \go$.
\end{prop}
\begin{proof}
	Since $S$ is a group bundle, its orbit space is precisely $S^{(0)} = \go$. The result now follows immediately from Corollaries \ref{cor:fellinnerexact} and 
	\ref{cor:orbitfibers}.
\end{proof}

Now let $q : \Cgh \to G$ be the Fell bundle over $G$ with $\Gamma_0(\go, \Cgh) = C^*(S, \B \vert_S)$ and $C^*(G, \Cgh) \cong C^*(E, \B)$, \`{a} la Buss and Meyer.
We will build our reduced iterated Fell bundle essentially by taking a quotient of this Fell bundle. To get us started, let $I$ denote the kernel of the natural quotient map 
$C^*(S, \B \vert_S) \to C_r^*(S, \B \vert_S)$. Then $I$ is a $C_0(\go)$-algebra, and it is not hard to see that $I(u)$ is the kernel of the quotient map 
$C^*(S_u, \B) \to C_r^*(S_u, \B)$ for each $u \in \go$. Indeed, the quotient map $\kappa : C^*(S, \B \vert_S) \to C_r^*(S, \B \vert_S)$ is clearly $C_0(\go)$-linear,
so it induces surjective homomorphisms $\kappa_u : C^*(S_u, \B) \to C_r^*(S_u, \B)$ for each $u \in \go$. The diagram
\[
	\xymatrix{
		C^*(S, \B \vert_S) \ar[r]^\kappa \ar[d] & C_r^*(S, \B \vert_S) \ar[d] \\
		C^*(S_u, \B) \ar[r]^{\kappa_u} & C_r^*(S_u, \B)
	}
\]
commutes, and it follows that $I_u = \ker \kappa_u$.

\begin{prop}
	If we view $C^*(S, \B \vert_S)$ as the unit $C^*$-algebra of the Fell bundle $q : \Cgh \to G$, then $I = \ker \kappa$ is a $G$-invariant ideal.
\end{prop}
\begin{proof}
	Suppose $P \in \hull(I)$, and let $x \in G$. Recall that $\B \vert_{S_{s(x)}}$ and $\B \vert_{S_{r(x)}}$ are equivalent via $\B \vert_{E_x}$, hence $C_r^*(S_{s(x)}, \B)$ 
	and $C_r^*(S_{r(x)}, \B)$ are Morita equivalent. Then $I(s(x))$ and $I(r(x))$ are clearly matched under the Rieffel correspondence $h_x$ as a consequence of the 
	equivalence theorem for reduced $C^*$-algebras \cite{sims-williams2013}. Hence $P \supseteq I(s(x))$ if and only if $h_x(P) \supseteq I(r(x))$, so $I$ is invariant.
\end{proof}

Since $I$ is a $G$-invariant ideal, there is a Fell bundle $\Cgh^I$ over $G$ associated to the quotient $C_r^*(S, \B \vert_S) = C^*(S, \B \vert_S) / I$ by 
\cite[Proposition 3.4]{dana-marius}. That is, $\Gamma_0(\go, \Cgh^I) = C_r^*(S, \B \vert_S)$. We aim to show that the isomorphism $C^*(G, \Cgh) \cong C^*(E, \B)$
of Buss and Meyer descends to an isomorphism $C_r^*(G, \Cgh^I) \cong C_r^*(E, \B)$. In other words, we claim that the map $\Gamma_c(E, \B) \to \Gamma_c(G, \Cgh^I)$
given by $\xi \mapsto \tilde{\xi}$, where
\[
	\tilde{\xi}(x)(e) = \xi(e),
\]
is isometric with respect to the reduced norms. We do so by carefully manipulating induced representations. 

Begin with a faithful representation $\pi$ of $A = \Gamma_0(\go, \B)$ on a Hilbert space $\Hil$. The
associated induced representation $\Ind_{S^{\scriptscriptstyle{(0)}}}^S \pi$ of $C_r^*(S, \B \vert_S)$ is faithful and acts on the Hilbert space 
$\mathsf{X} = \overline{\Gamma_c(S, \B \vert_S) \odot \Hil}$. Recalling that $C_r^*(S, \B \vert_S) = \Gamma_0(G, \Cgh^I)$, we can now induce up to $C_r^*(G, \Cgh^I)$
to obtain a faithful representation $\Ind_{\go}^G (\Ind_{S^{\scriptscriptstyle{(0)}}}^S \pi)$ on $\overline{\Gamma_c(G, \Cgh^I) \odot \mathsf{X}}$. On the other hand,
we can invoke \cite[Theorem 2.4]{IW} and build a faithful representation of $C_r^*(E, \B)$ using induction in stages. That is, we form the representation
$\Ind_S^E (\Ind_{S^{\scriptscriptstyle{(0)}}}^S \pi)$ of $C_r^*(E, \B)$ on the completion of $\Gamma_c(E, \B) \odot \mathsf{X}$. We claim that there is a unitary
$U : \overline{\Gamma_c(E, \B) \odot \mathsf{X}} \to \overline{\Gamma_c(G, \Cgh^I) \odot \mathsf{X}}$ that implements the desired isomorphism between $C_r^*(E, \B)$
and $C_r^*(G, \Cgh^I)$.

\begin{prop}
	The map 
	\[
		U_0 : \Gamma_c(E, \B) \odot \Gamma_c(S, \B \vert_S) \odot \Hil \to \Gamma_c(G, \Cgh^I) \odot \Gamma_c(S, \B \vert_S) \odot \Hil
	\]
	given by
	\[
		U_0(\xi \otimes \sigma \otimes h) = \tilde{\xi} \otimes \sigma \otimes h
	\]
	extends to a unitary $U : \overline{\Gamma_c(E, \B) \odot \mathsf{X}} \to \overline{\Gamma_c(G, \Cgh^I) \odot \mathsf{X}}$, which spatially implements an
	isomorphism between $C_r^*(E, \B)$ and $C_r^*(G, \Cgh^I)$.
\end{prop}
\begin{proof}
	Clearly $U_0$ has dense range, so we just need to check that it preserves inner products. In fact, it really only suffices to check that
	\begin{equation}
	\label{eq:innerproduct}
		\hip{\xi}{\zeta}_{C_r^*(S, \B)} = \hip{\tilde{\xi}}{\tilde{\zeta}}_{C_r^*(S, \B)}
	\end{equation}
	for all $\xi, \zeta \in \Gamma_c(E, \B)$. To see why this condition is enough, observe that it implies
	\begin{align*}
		\ip{U_0(\xi \otimes \sigma \otimes h)}{U_0(\zeta \otimes \tau \otimes k)} &= \ip{\tilde{\xi} \otimes \sigma \otimes h}{\tilde{\zeta} \otimes \tau \otimes k} \\
			&= \ip{\Ind_{S^{\scriptscriptstyle{(0)}}}^S \pi (\hip{\tilde{\zeta}}{\tilde{\xi}})(\sigma \otimes h)}{\tau \otimes k} \\
			&= \ip{\hip{\tilde{\zeta}}{\tilde{\xi}} \cdot \sigma \otimes h}{\tau \otimes k} \\
			&= \ip{\Ind_{{S^{\scriptscriptstyle{(0)}}}}^S \pi (\hip{\zeta}{\xi})(\sigma \otimes h)}{\tau \otimes k} \\
			&= \ip{\xi \otimes \sigma \otimes h}{\zeta \otimes \tau \otimes k},
	\end{align*}
	so $U_0$ preserves inner products. Therefore, we will focus on proving \eqref{eq:innerproduct}.
	
	Let $\xi, \zeta \in \Gamma_c(E, \B)$. Then for all $t \in S$ we have
	\begin{align*}
		\hip{\zeta}{\xi}(t) &= \int_E \zeta(e)^* \xi(et) \, d\nu_u(e) \\
			&= \int_E \zeta(e^{-1})^* \xi(e^{-1}t) \, d\nu^u(e) \\
			&= \int_G \int_S \zeta((et')^{-1})^* \xi((et')^{-1}h) \, d\mu^{s(e)}(t') \, d\lambda^u(j(e)) \\
			&= \int_G \int_S \zeta(t'^{-1} e^{-1})^* \xi(t'^{-1} e^{-1} t) \, d\mu^{s(e)}(t') \, d\lambda^u(j(e)).
	\end{align*}
	On the other hand,
	\begin{align*}
		\hip{\tilde{\zeta}}{\tilde{\xi}}(u)(t) &= \left(\int_G \tilde{\zeta}(x)^* \tilde{\xi}(x) \, d\lambda_u(x) \right)(t) \\
			&= \int_G \bigl( \tilde{\zeta}(x)^* \tilde{\xi}(x) \bigr)(t) \, d\lambda_u(x),
	\end{align*}
	where
	\begin{align*}
		\bigl(\tilde{\zeta}(x)^* \tilde{\xi}(x)\bigr)(t) &= \int_S \tilde{\zeta}(x)^*(et') \tilde{\xi}(x)((et')^{-1} t) \, d\mu^{s(e)}(t') \\
			&= \int_S \zeta(t'^{-1} e^{-1})^*\xi(t'^{-1} e^{-1} t) \, d\mu^{s(e)}(t').
	\end{align*}
	Putting everything together, the result follows. It is fairly clear that $U$ implements the desired isomorphism, so $C_r^*(G, \Cgh^I) \cong C_r^*(E, \B)$.
\end{proof}

Now we arrive at our intended application---showing that extensions of exact groupoids are exact. Suppose
\[
	\go \to S \to E \to G \to \go
\]
is an extension of groupoids, and let $p : \B \to E$ be a Fell bundle. Put $A = \Gamma_0(\go, \B)$, and let $I \subseteq A$ be an $E$-invariant ideal. Then we have 
an associated Fell bundle $\B_I$ over $E$, and $C_r^*(E, \B_I)$ and $C_r^*(S, \B_I \vert_S)$ sit inside $C_r^*(E, \B)$ and $C_r^*(S, \B \vert_S)$, respectively,
as ideals. We also have a Fell bundle $q : \D \to G$ with $\Gamma_0(G, \D) = C_r^*(S, \B \vert_S)$, and we claim that $J = C_r^*(S, \B_I \vert_S)$ is a $G$-invariant
ideal of its unit $C^*$-algebra. Well, if $x \in G$ and $e \in E_x$, then
\[
	h_e(I(s(x))) = I(r(x))
\]
since $I$ is invariant. Proposition \ref{prop:rieffel} implies that if $h_x$ is the Rieffel correspondence induced by $\D_x$, then we have
\[
	h_x(C_r^*(S_{s(x)}, \B_I)) = C_r^*(S_{r(x)}, \B_I),
\]
or $h_x(J(s(x))) = J(r(x))$. Thus $J$ is a $G$-invariant ideal.

Now assume $S$ is exact. Then the sequence
\[
	0 \to C_r^*(S, \B_I \vert_S) \to C_r^*(S, \B) \to C_r^*(S, \B^I \vert_S) \to 0
\]
is exact. Since $J = C_r^*(S, \B_I \vert_S)$ is a $G$-invariant ideal, we get a sequence
\[
	0 \to C_r^*(G, \D_J) \to C_r^*(G, \D) \to C_r^*(G, \D^J) \to 0,
\]
which is exact provided $G$ is exact. It is not hard to convince oneself that $\D_J$ and $\D^J$ are precisely the Buss-Meyer Fell bundles coming from $\B_I$ and $\B^I$, 
respectively, so we get a commuting diagram
\[
	\xymatrix{
		0 \ar[r] & C_r^*(E, \B_I) \ar[r] \ar[d] & C_r^*(E, \B) \ar[r] \ar[d] & C_r^*(E, \B^I) \ar[r] \ar[d] & 0 \\
		0 \ar[r] & C_r^*(G, \D_I) \ar[r] & C_r^*(G, \D) \ar[r] & C_r^*(G, \D^J) \ar[r] & 0
	}
\]
where the vertical arrows are isomorphisms. Since the bottom row is exact whenever $S$ and $G$ are exact, it follows that the top row is also exact. Thus we have proven
the following result.

\begin{thm}
	Suppose $\go \to S \to E \to G \to \go$ is an extension of second countable, locally compact Hausdorff groupoids. If $S$ and $G$ are exact, then so is $E$.
\end{thm}

A groupoid $G$ is said to be \emph{inner exact} if for any open, invariant set $U \subseteq \go$, the sequence
\[
	0 \to C_r^*(G \vert_U) \to C_r^*(G) \to C_r^*(G \vert_F) \to 0
\]
is exact, where $F = \go \backslash U$. If we take $\B$ to be the trivial Fell line bundle over $E$ (so that $C_r^*(E, \B) = C_r^*(E)$), then the same arguments as above can be used 
to establish a result for inner exactness of groupoid extensions.

\begin{thm}
	Suppose $\go \to S \to E \to G \to \go$ is an extension of second countable, locally compact Hausdorff groupoids. If $S$ is inner exact and $G$ is exact, then $E$ is inner exact.
\end{thm}

\bibliographystyle{amsplain}
\bibliography{/Users/slalonde/Documents/Research/SharedFiles/GroupBib.bib}

\end{document}